\documentclass[final,3p]{elsarticle}

\usepackage{lineno}
\usepackage[colorlinks=true,breaklinks=true,pdftex]{hyperref}
\modulolinenumbers[1]

\usepackage{fleqn}
\usepackage{amsmath}
\usepackage{amsthm}
\usepackage{amsfonts}
\usepackage{amssymb}
\usepackage{epsfig}
\usepackage{makeidx}
\usepackage{relsize}
\usepackage[export]{adjustbox}
\usepackage{graphics}
\usepackage{epstopdf}
\usepackage{subfigure}
\usepackage{graphicx}
\usepackage{xcolor}
\usepackage[normalem]{ulem}
\usepackage{algorithm}
\usepackage{algpseudocode}
\usepackage{enumerate}
\usepackage{dsfont}
\usepackage{multirow}
\usepackage[title]{appendix}

\usepackage{tabularx}
\usepackage{ragged2e}
\newcolumntype{M}[1]{>{\centering\arraybackslash}m{#1}}
\usepackage{tikz}
\usetikzlibrary{plotmarks}
\usetikzlibrary{arrows}
\usetikzlibrary{shapes}

\newcommand{\cC}{\mathcal{C}}
\newcommand{\cF}{\mathcal{F}}
\newcommand{\cS}{\mathcal{S}}

\newcommand{\R}{\mathbb{R}}
\newcommand{\N}{\mathbb{N}}
\newcommand{\Z}{\mathbb{Z}}

\newcommand{\lra}{\longrightarrow}

\newcommand{\lZn}{\ell(\Z,\R^n)}
\newcommand{\liZn}{\ell_\infty(\Z,\R^n)}

\DeclareMathOperator*{\Span}{span}

\newcommand{\espan}[1]{\Span\left\{#1\right\}}

\theoremstyle{remark}
\newtheorem{thm}{Theorem}
\newtheorem{lem}[thm]{Lemma}
\newtheorem{teo}[thm]{Theorem}
\newtheorem{cor}[thm]{Corolary}
\newtheorem{prop}[thm]{Proposition}

\theoremstyle{remark}
\newtheorem{defi}[thm]{Definition}
\newtheorem{rmk}[thm]{Remark}

\journal{arXiv}

\biboptions{authoryear}

\begin{document}

\begin{frontmatter}
    
    \title{A uniform non-linear subdivision scheme \\ reproducing polynomials at any non-uniform grid}
    
    
    \author{Sergio L\'opez-Ure\~na\corref{cor}}
    \ead{sergio.lopez-urena@uv.es}
    \cortext[cor]{Corresponding author}
    \address{Dept. de Matem\`atiques, Universitat de Val\`encia, Doctor Moliner Street 50, 46100 Burjassot, Valencia, Spain.}
    
    %
    %
    \begin{abstract}
        In this paper, we introduce a novel non-linear uniform subdivision scheme for the generation of curves in $\R^n$, $n\geq2$. This scheme is distinguished by its capacity to reproduce second-degree polynomial data on non-uniform grids without necessitating prior knowledge of the grid specificities. Our approach exploits the potential of annihilation operators to infer the underlying grid, thereby obviating the need for end-users to specify such information. We define the scheme in a non-stationary manner, ensuring that it progressively approaches a classical linear scheme as the iteration number increases, all while preserving its polynomial reproduction capability.
		
		The convergence is established through two distinct theoretical methods. Firstly, we propose a new class of schemes, including ours, for which we establish $\mathcal{C}^1$ convergence by combining results from the analysis of quasilinear schemes and asymptotically equivalent linear non-uniform non-stationary schemes.
		Secondly, we adapt conventional analytical tools for non-linear schemes to the non-stationary case, allowing us to again conclude the convergence of the proposed class of schemes.
		
		We show its practical usefulness through numerical examples, showing that the generated curves are curvature continuous.
    \end{abstract}
    
    
\end{frontmatter}


\section{Introduction}

Recursive subdivision is a well-established technique for creating smooth geometric models from discrete initial data. It has found extensive application in Computer-Aided Design (CAD) thanks to its exceptional performance and straightforward implementation (e.g., see \cite{Dyn92,Ma05}).

In this paper, our focus is on generating curves in $\R^n$, $n\geq2$. To achieve this, we will employ the following type of subdivision schemes.

\begin{defi}[Vector-valued subdivision scheme] \label{defi:vector_subdivision}
	For a given initial data $f^0\in\lZn$, so that $f^0_i\in\R^n$ $\forall i\in\Z$, more data is generated by the recursive formula $f^{k+1} = S_kf^k$, $k\in\N_0 := \N \cup \{0\}$, where $S_k:\lZn\lra \lZn$ is a \emph{subdivision operator} defined as
\[
(S_{k}f)_{2i} = \Psi_{2i}^k(f_{i-p},\ldots,f_{i+p+1}), \quad (S_{k}f)_{2i+1} = \Psi_{2i+1}^k(f_{i-p},\ldots,f_{i+p+1}), \qquad i\in\Z, \quad f\in\lZn,
\]
for some $p\in\N_0$ and some \(\Psi^k_i : (\R^n)^{2p+2}\lra\R^n\), $i\in\Z$. The \emph{subdivision scheme} is identified as the sequence of subdivision operators \(\cS = \{S_k\}_{k\in\N_0}\).
The scheme is \emph{linear} iff each $\Psi^k_i$ is a linear function. The scheme is \emph{uniform} iff $\Psi^k_{2i} = \Psi^k_0$ and $\Psi^k_{2i+1} = \Psi^k_1$, $\forall i\in\Z$, $\forall k\in\N_0$. The scheme is \emph{stationary} iff $\Psi^k_i = \Psi^0_i$, $\forall i\in\Z$, $\forall k\in\N_0$.
We say that the scheme is \emph{based on a scalar-valued subdivision scheme} (or simply \emph{scalar-based}) iff there exists $\psi^k_i : \R^{2p+2}\lra\R$ such that
\[
	\Psi^k_i(f_{i-p},\ldots,f_{i+p+1}) = (\psi^k_i((f_{i-p})_j,\ldots,(f_{i+p+1})_j))_{j=1}^n, \qquad \forall f_{i-p},\ldots,f_{i+p+1}\in\R^n, \ \forall i\in\Z, \ \forall k\in\N_0.
\]
\end{defi}

The properties of linear uniform subdivision schemes have been extensively studied, including non-stationary schemes, which are valuable for achieving exact geometric representations of conic sections through exponential polynomial reproduction (see \cite{DLL03,CR11}).

Non-uniform linear schemes have also been the subject of previous research (as in \cite{DLY14}). However, their practical application in CAD may be limited, necessitating knowledge of the underlying non-uniform grid or requiring end-users to adjust certain shape parameters along the grid (e.g., see \cite{KV02}). To address this issue, some scholars have suggested estimating these parameters from existing data, either at the initial step (as in \cite{BCR11}) or at each step (as in \cite{JYY21,DFH09}), with the latter approach leading to uniform non-linear schemes. A similar strategy has been employed in this work: we consider a family of non-uniform linear schemes and replace the grid-dependent parameters with non-linear functions that solely dependent on the data, leading to a uniform non-linear scheme.

In this paper, we consider subdivision schemes as general as in Definition \ref{defi:vector_subdivision}. In the absence of some the adjectives ``linear'', ``uniform'', ``stationary'' and ``scalar-based'', we do not assume them.

In the following, we introduce some notation and basic definitions. For any $f\in\lZn$, we denote $\nabla f := (\nabla f_{i})_{i\in\Z} \in \lZn$, $\nabla f_{i} := f_{i+1} - f_{i} \in \R^n$. For bounded sequences, we consider
\[
\|f\|_\infty := \sup_{i\in\Z} |f_i|, \quad f\in\liZn,
\]
for some fixed $|\cdot|$ norm in $\R^n$.
Abusing the notation, we write $\langle a,b \rangle := \sum_{i} a_i b_i \in\R^m$, for $a$ and $b$ having the same length, $a_i\in\R$, $b_i\in\R^m$, with $m\in\N$.

\begin{defi}[Convergence]
We say that a sequence $\{f^{k}\}_{k\in\N_0}\subset \lZn$ \emph{converges uniformly} to a continuous function $F: \R \lra \R^n$ iff for any compact interval $I\subset\R$,
\begin{equation}
\lim_{k\to+\infty} \sup_{i \in \Z \cap 2^k I} |f^k_{i} - F(i 2^{-k})|=0.
\end{equation}
A subdivision scheme $\cS$ is \emph{uniformly convergent} iff, for any $f^0\in \lZn$, the sequence $\{f^{k}\}_{k\in\N_0}$ defined recursively by $f^{k+1} = S_{k}f^{k}$ converges uniformly to a continuous function, denoted by $S^\infty f^0$, and $S^\infty \neq 0$. A scheme is $\cC^\mu$, $\mu\in(0,+\infty)$, iff it is convergent and $S^\infty f^0 \in \cC^\mu(\R,\R^n)$, $\forall f^0\in \lZn$ (Hölder regularity is considered here). A scheme is $\cC^{\mu-}$ iff it is $\cC^{\mu-\epsilon}$ for any $\epsilon\in(0,\mu)$.
\end{defi}

The concept of reproduction refers to the ability of a subdivision scheme to exactly generate a specific family of functions, typically polynomials. It has been thoroughly examined, but to the best of our knowledge, it has been barely addressed in the context of non-uniform grids (see \cite{Levin03}). Consequently, we propose the following reproduction definition, which is compatible with non-uniform grids.

\begin{defi}[Reproduction]
For each $k\in\N_0$, let $\xi^k = (\xi^k_i)_{i\in\Z} \subset \R$ be a grid, so that $\xi^k_i < \xi^k_{i+1}$.
Let $\cF\subset\cC(\R,\R^n)$ be a set of functions. We say that a subdivision operator $S_{k}$ \emph{reproduces} $\cF$ for $(\xi^{k},\xi^{k+1})$ iff
\[
S_k F|_{\xi^k} = F|_{\xi^{k+1}}, \qquad \forall F\in\cF,
\]
where $F|_{\xi^k} = (F(\xi^k_i))_{i\in\Z}$.
If this holds true for every $k\in\N_0$, we say that the subdivision scheme $\cS$ \emph{stepwise reproduces} $\cF$ for $\{\xi^{k}\}_{k\in\N_0}$.

We say that $\cS$ \emph{reproduces} $\cF$ on $\xi^0$ iff there exists some parametrization of $\R$, say $\sigma:\R\lra\R$ strictly increasing, such that the subdivision process converges to $F\circ\sigma$ when the input data is $f^{0} = F|_{\xi^0}$, $\forall F\in\cF$.
\end{defi}

The typical definition of reproduction does not consider a parametrization $\sigma$. This is because it is provided for uniform schemes, where $\sigma$ is implicitly selected as the identity function and $\xi^0$ is a uniform (i.e., equispaced) grid.

\begin{prop} \label{prop:step_implies_limit}
Let $\cF\subset\cC(\R,\R^{n})$ be and let $\{\xi^{k}\}_{k\in\N_0}$ be a uniformly convergent sequence. 
If $\cS$ stepwise reproduces $\cF$ for $\{\xi^{k}\}_{k\in\N_0}$, then $\cS$ reproduces $\cF$ on $\xi^{0}$.

In addition, $S^\infty (F|_{\xi^0}) = F\circ \sigma$, where $\sigma:\R\lra\R$ is the limit function of $\{\xi^{k}\}_{k\in\N_0}$.
\end{prop}
\begin{proof}
Consider $I\subset\R$ an arbitrary compact interval and $F\in\cF$ an arbitrary function of the space being reproduced.
Given the stepwise reproduction, we know that starting from any $f^{0} = F|_{\xi^0}$ the generated iterates are $f^{k} = F|_{\xi^k}$. Thus,
\[
|f^k_i - (F\circ \sigma)(i 2^{-k})| =
|F(\xi^{k}_{i})- F(\sigma(i 2^{-k}))|.
\]
Given $\epsilon>0$, since $F$ is uniformly continuous on $I$, there exists $\delta>0$ such that
\[
|x-y|<\delta \quad \lra \quad |F(x)-F(y)|<\epsilon.
\]
Since $\{\xi^{k}\}_{k\in\N_0}$ converges uniformly to $\sigma$, there exists $k_{0}\in\N_{0}$ such that
\[
|\xi^k_i - \sigma(i 2^{-k})| < \delta, \qquad \forall i\in \Z \cap 2^{k}I, \quad \forall k\geq k_{0}.
\]
Combining these results, we get
\[
|f^k_i - (F\circ \sigma)(i 2^{-k})|<\epsilon, \quad  \forall i\in \Z \cap 2^{k}I, \quad \forall k\geq k_{0}.
\]
Hence, $\{f^{k}\}_{k\in\N_0}$ converges uniformly to $F\circ \sigma$ on $I$. Since $I$ and $F$ were arbitrarily chosen, we conclude that $\cS$ reproduces $\cF$ on $\xi^{0}$.
\end{proof}

Thanks to \cite{DLL03,CR11}, the algebraic conditions for a linear uniform scheme to reproduce exponential polynomials are well-established. For instance, to reproduce a circle, the number of sampling points is required to define a reproducing linear non-stationary scheme. In the linear case, the consideration of non-stationary schemes is mandatory to reproduce exponential polynomials.

In later studies (see \cite{DL19,LV21}), we investigated the possibility of defining a single non-linear but stationary scheme with the capability to simultaneously reproduce multiple exponential polynomial spaces. In the previous example of a circle, this approach eliminates the need to know the number of sampling points. This type of schemes was accomplished by leveraging the existence of annihilation operators for exponential polynomials (see \cite{DLL03,CLR22}).

In this work, we harness annihilation operators to develop a non-linear but uniform subdivision scheme with the ability to reproduce second-degree polynomials even if they are sampled on non-uniform grids. This scheme can handle many non-uniform grids without the need for prior knowledge of those grids. No preprocessing steps are performed. This property is precisely stated in Theorem \ref{teo:scheme_reproduction}.

The capability to represent parabolas in CAD applications, afforded by second-degree polynomials reproduction, holds little significance.
This research pretends to mark an initial step towards achieving the reproduction of exponential polynomials (and in particular, conic sections) on non-uniform grids, even when the grid information is not known in advance.
Such reproduction capabilities can be particularly valuable in practical scenarios where users provide a set of points, here denoted as $f^0$, representing a common curve like a conic section, without specifying the underlying parametrization or grid spacings. Consequently, these curves could be reproduced on a wide range of non-uniform grids without burdening the user with providing such information.

More in detail, we deduce and study a subdivision scheme that can reproduce the set of functions whose coordinates are second-degree polynomials,
\[ \Pi_2^n := \{ F=(F_1,\ldots,F_n):\R\lra\R^n, \quad F_i \in \Pi_2, \quad \forall i=1,\ldots,n\}.\]

\begin{rmk}
	From now on, let us consider a sequence of grids $\{\xi^k\}_{k\in\N_0}$, so that $\xi^k_i < \xi^k_{i+1}$, $\forall i\in\Z$, $\forall k\in\N_0$. We measure the uniformity of the grid $\xi^k$ by the following quantities:
\[
\alpha_i^k := \frac{\nabla \xi^k_{i-1}}{\nabla \xi^k_i}, \quad \beta_{i}^{k}:=\frac{\nabla \xi^k_{i+1}}{\nabla \xi^k_i}, \qquad i\in\Z.
\]
Observe that $\alpha_i^k,\beta_i^k\in(0,+\infty)$, $\beta_i^k = 1/\alpha_{i+1}^k$, and that a uniform grid is characterized by $\alpha_i^k=1$, $\forall i\in\Z$.
\end{rmk}

Second-degree polynomials satisfy the following annihilation property.
\begin{prop}[Annihilation property] \label{prop:annihilation}
	For any grid $\xi^k$ and for all $i\in\Z$, $F\in\Pi_2^n$, $n\geq 1$,
\begin{equation} \label{prop:eq:annihilation}
	A_i^k (F(\xi^k_{i}) - F(\xi^k_{i-1})) - (F(\xi^k_{i+1})-F(\xi^k_i)) + B_i^k (F(\xi^k_{i+2})-F(\xi^k_{i+1})) = 0\in\R^n,
\end{equation}
where
\begin{equation}
\label{eq:alpha_beta_ik}
A_i^k:=\frac{\beta_i^k+1}{\alpha_i^k (\alpha_i^k+\beta_i^k+2)}, \quad B_i^k:=\frac{\alpha_i^k+1}{\beta_i^k (\alpha_i^k+\beta_i^k+2)}.
\end{equation}
\end{prop}
\begin{proof}
	By the linearity of \eqref{prop:eq:annihilation} respect to $F$, and since $\Pi_2^n$ is a tensor product space, it is enough to check this property on a basis of $\Pi_2^1$, say $\{1,x,x^2\}$. It can be confirmed by simplifying the algebraic expressions that result from replacing \eqref{eq:alpha_beta_ik} in \eqref{prop:eq:annihilation} for each $F\in\{1,x,x^2\}$. We provide a Mathematica notebook file (see Section Reproducibility) where this verification, and others, are performed.
\end{proof}

The Lagrange interpolation formula can be employed to readily derive a linear non-uniform subdivision operator that reproduces $\Pi_3^n$ for $(\xi^k,\xi^{k+1})$. It forms the basis of our proposal for a uniform non-linear scheme, as it did for other schemes (see \cite{DFH04, DFH09}).
\begin{defi}[Lagrange subdivision scheme] \label{defi:lagrange}
Given $\{(\alpha^k_i,\beta^k_i)\}_{i\in\Z}$, $\alpha^k_i,\beta^k_i\in(0,+\infty)$, $k\in\N_0$, we consider the subdivision operators
\begin{align}
L_{k}&:\lZn \lra \lZn \nonumber \\
(L_{k} f)_{2i+l}& := \Lambda_{l}[\alpha^k_{i},\beta^k_{i}](f_{i-1},f_{i},f_{i+1},f_{i+2}), \quad l=0,1, \ i\in\Z, \nonumber \\
	\label{def:Lambda}
	\Lambda_{0}[\alpha,\beta](f_{-1},f_{0},f_{1},f_{2}) &:= \sum_{j=-1}^2 a_{j}(\alpha,\beta) \ \cdot \ f_{j}, \quad
	\Lambda_{1}[\alpha,\beta](f_{-1},f_{0},f_{1},f_{2}) := \sum_{i=-1}^2 a_{1-j}(\beta,\alpha) \ \cdot \ f_{j},
\\
 \label{eq:ai}	\begin{split}
		a_{j}:&(0,+\infty)^{2} \lra \R, \quad j=-1,0,1,2, \\
		a_{-1}(\alpha,\beta) &= -\frac{3 (4 \beta+3)}{64 \left(\alpha^2+\alpha\right) (\alpha+\beta+1)}, \quad
		a_{0}(\alpha,\beta) = \frac{3 (4 \alpha+1) (4 \beta+3)}{64 \alpha (\beta+1)},\\
		a_{2}(\alpha,\beta) &= -\frac{3 (4 \alpha+1)}{64 \left(\beta^2+\beta\right)(\alpha+\beta+1) }, \quad
		a_{1}(\alpha,\beta) = \frac{(4 \alpha+1) (4 \beta+3)}{64  \beta(\alpha+1)}.
	\end{split}
\end{align}
\end{defi}

We have defined $L_k$ using $\Lambda_0$ and $\Lambda_1$ to emphasize that the dependence on the grid and the iteration $k$ is exclusively given by $\alpha^k_{i}$ and $\beta^k_{i}$. Our proposal of subdivision scheme, which will be defined in Section \ref{sec:new_scheme}, consists of replacing $\alpha^k_{i},\beta^k_{i}$ by some non-linear functions depending only on the data $f$.

The reason for considering $L_k$ to reproduce $\Pi_3^n$ instead of $\Pi_2^n$, our target, is that $L_k$ is uniquely determined by that 4-point stencil and its reproducing condition. The choice of 4 points is due to it being the minimum number required to locally use Proposition \ref{prop:annihilation} for computing $\alpha^k_{i}$ and $\beta^k_{i}$ from $f^k_{i-1}, f^k_{i}, f^k_{i+1}, f^k_{i+2}$, as we will see in Section \ref{sec:new_scheme}.

\begin{prop} \label{prop:lagrange_reproduces}
Let $\xi^{k}$ be a grid and define $\xi^{k+1}$ by applying the Chaikin's subdivision operator, i.e.,
	\[
	\xi^{k+1}_{2i} = \frac34 \xi^k_{i} + \frac14 \xi^k_{i+1}, \qquad \xi^{k+1}_{2i+1} = \frac14 \xi^k_{i} + \frac34 \xi^k_{i+1}, \qquad i\in\Z.
	\]
The subdivision operator in Definition \ref{defi:lagrange} reproduces $\Pi^{n}_{3}$ for $(\xi^k,\xi^{k+1})$.
\end{prop}
\begin{proof}
The Chaikin's operator is strictly monotonicity preserving, implying that $\xi^{k+1}$ is strictly increasing because $\xi^{k}$ is.
$L_{k}$ was computed using the Lagrange interpolation formula, ensuring reproduction. The steps taken to derive it are outlined here, and for a direct check of the reproduction property, we refer to the Mathematica file (see Section Reproducibility).

First, given the data $f_{j} = F(\xi_{j})$, $j=-1,0,1,2$, attached to some nodes $\xi_{-1},\xi_{0},\xi_{1},\xi_{2}$, there is a unique cubic polynomial interpolating them, that can be expressed as a linear combination of the Lagrange basis functions:
\[
F(x) = \sum_{i=-1}^{2} f_{i} \prod_{\substack{j=-1 \\ j \neq i}}^2 \frac{x-\xi_{j}}{\xi_{i}-\xi_{j}}.
\]
We are only concerned with two values of $x$, $x = \frac{3}{4} \xi_0 + \frac{1}{4} \xi_1$ and $x = \frac{1}{4} \xi_0 + \frac{3}{4} \xi_1$, for deriving the subdivision rules. Here, we present the first choice of $x$, as the outcome for the second one should mirror the first, with the order of data and nodes reversed.

To simplify the expression, we denote $\alpha := \nabla \xi_{-1} /\nabla \xi_0$, $\beta:=\nabla \xi_{1} / \nabla \xi_0$, and
\[
F(x) = \sum_{i=-1}^{2} a_{i}f_{i}, \qquad a_{i} := \prod_{\substack{j=-1 \\ j \neq i}}^2 \frac{x-\xi_{j}}{\xi_{i}-\xi_{j}}.
\]
To express $a_i$ in terms of $\alpha$ and $\beta$, we compute the following:
\begin{align*}
\frac{\xi_{2} - \xi_{-1}}{\nabla\xi_{0}} &= \frac{\nabla\xi_{1}+\nabla \xi_{0}+\nabla \xi_{-1}}{\nabla \xi_{0}} = \beta + 1 + \alpha, \quad
\frac{\xi_{1} - \xi_{-1}}{\nabla\xi_{0}}= 1 + \alpha, \quad
\frac{\xi_{2} - \xi_{0}}{\nabla\xi_{0}}  = \beta + 1, \\
\frac{x - \xi_{0}}{\nabla\xi_{0}} &= \frac{ - \frac14 \xi_{0} + \frac14\xi_{1}}{\nabla\xi_{0}} = \frac14, \quad
\frac{x - \xi_{-1}}{\nabla\xi_{0}} = \frac{ x - \xi_{0} + \nabla\xi_{-1}}{\nabla\xi_{0}} = \frac14 + \alpha, \quad
\frac{x - \xi_{1}}{\nabla\xi_{0}} = \frac{ \frac34 \xi_{0} - \frac34\xi_{1}}{\nabla\xi_{0}} = -\frac34, \\
\frac{x - \xi_{2}}{\nabla\xi_{0}} &=\frac{ x - \xi_{1} -  \nabla\xi_{1}}{\nabla\xi_{0}} = -\frac34 - \beta.
\end{align*}
Then, it is straightforward to see that
\begin{align*}
 a_{-1} &= \prod_{j=0}^2 \frac{x-\xi_{j}}{\xi_{-1}-\xi_{j}} = \prod_{j=0}^2 \frac{(x-\xi_{j})/\nabla\xi_{0}}{(\xi_{-1}-\xi_{j})/\nabla\xi_{0}} 
 = \frac{\frac14 (-\frac34)(-\frac34 - \beta)}{-\alpha(-1-\alpha)(-1-\alpha-\beta)} = -\frac{3(3 + 4\beta)}{64\alpha(1+\alpha)(1+\alpha+\beta)}.
\end{align*}
We can proceed similarly with the other coefficients to obtain \eqref{eq:ai}.
\end{proof}

\begin{rmk} \label{rmk:lagrange_affine}
	It can be verified that $\sum_{i=-1}^2 a_{i}(\alpha,\beta) = 1$, a consequence of the reproduction of constant functions, which are included in $\Pi_3^n$. Additionally, the scheme is scalar-based. These two properties imply that $L_k$ commutes with affine transformations, i.e.,
	\[
		L_k(\mathcal{A}_{M,c} f) = \mathcal{A}_{M,c} (L_k f), \quad \mathcal{A}_{M,c} f := (M f_i + c)_{i\in\Z}, \quad \forall M\in\R^{n\times n}, \ \forall c \in\R^n, \ \forall f\in\lZn.
	\]
	This property is useful in CAD applications and it is easily verified by direct computation.
\end{rmk}


\begin{rmk} \label{rmk:dyn_scheme}
For uniform (i.e. equispaced) grids, where $\alpha^k_{i}=\beta^k_{i}=1$, the Lagrange subdivision operators coincide with the one studied in \cite{DFH04}, whose rules are
\begin{align*}
	\Lambda_0[1,1](f_{-1},f_{0},f_{1},f_{2}) = \frac1{128}(-7f_{-1} +105f_0 +35f_1 -5f_2), \\
	\Lambda_1[1,1](f_{-1},f_{0},f_{1},f_{2}) = \frac1{128}(-5f_{-1} +35f_0 +105f_1 -7f_2).
\end{align*}
This linear scheme is $\cC^2$.
\end{rmk}

We have introduced all the necessary components to establish a single non-linear and uniform subdivision scheme capable of reproducing $\Pi_2^n$ on a variety of grids. This scheme will be defined next, in Section \ref{sec:new_scheme}. It is important to note that the scheme will be defined in a non-stationary manner, which is not a requisite for achieving reproduction capability but proves to be highly advantageous for demonstrating convergence. We adopt two different approaches to prove convergence, as we find them interesting from a theoretical development perspective. In Section \ref{sec:convergence_asymptotic}, we combine the results of \emph{asymptotic equivalence} from \cite{DLY14} for linear non-uniform non-stationary schemes with the \emph{quasilinear} concept of \cite{CDM03}. This synthesis gives rise to a new class of subdivision schemes that demonstrate convergence and, with the addition of an easily met requirement, achieve $\cC^1$ continuity. The last is applicable to our scheme.
In Section \ref{sec:convergence_nonlinear}, we analyse convergence using standard analytical tools for non-linear schemes, demonstrating their applicability in the non-stationary case by easily adapting them. Section \ref{sec:numerical} illustrates the performance of our proposal in several numerical examples, highlighting its ability to produce curvature-continuous curves and good-looking shapes. Our work concludes with final remarks and considerations for the future in Section \ref{sec:conclusions}.

\section{Definition of the new subdivision scheme. Analysis of reproduction} \label{sec:new_scheme}

\begin{defi} \label{defi:scheme}
For a fixed $\rho\geq0$ and $n\geq 2$, we consider the uniform non-linear non-stationary subdivision scheme defined by the rules $\Psi^k_0,\Psi^k_1:(\R^n)^4\lra\R^n$,
\begin{align} \label{eq:nonlinear_rules}
\begin{split}
\Psi^k_l(f_{-1},f_{0},f_{1},f_{2}) &:= \Lambda_{l}[((\alpha^k,\beta^k)\circ(A,B))(\nabla f_{-1},\nabla f_{0},\nabla f_{1})](f_{-1},f_{0},f_{1},f_{2}), \quad l=0,1,
\end{split}
\end{align}
where $\Lambda_{0},\Lambda_{1}$ are defined in \eqref{def:Lambda} and
\begin{align}
	\nonumber \alpha^k,\beta^k:(0,&+\infty)^2\lra[(1-2^{-k}\rho)^{-1},1+2^{-k}\rho], \qquad A,B:(\R^n)^3\lra(0,+\infty)\\
\begin{split} \label{eq:alpha_beta_k}
	\alpha^k(A,B) &:= \max\left\{(1+2^{-k}\rho)^{-1},\min\left\{1+2^{-k}\rho,\frac{1}{A+\frac{\sqrt{A (A+1) B (B+1)}}{B+1}}\right\}\right\}, \\
	\beta^k(A,B)&:= \alpha^k(B,A),
\end{split}\\
\label{eq:AB}
\begin{split}
		A(\delta_{-1},\delta_{0},\delta_{1}) &:= \begin{cases}
			\frac 12, & \text{if } \text{num}\cdot\text{den} = 0, \\
			\text{abs}(\frac{
				\text{num}
			}{
				\text{den}
			}), & \text{otherwise}.
			\end{cases}
			, \quad
			B(\delta_{-1},\delta_{0},\delta_{1}) := A(\delta_{1},\delta_{0},\delta_{-1})\\
			\text{num} &= \det\begin{pmatrix}
				(\delta_{-1})^T \delta_0 & (\delta_{-1})^T \delta_{1}\\
				(\delta_{1})^T \delta_0 & |\delta_{1}|_2^2
			\end{pmatrix},
			\quad
			\text{den} = 
			\det\begin{pmatrix}
				|\delta_{-1}|_2^2 & (\delta_{-1})^T \delta_{1}\\
				(\delta_{-1})^T \delta_{1} & |\delta_{1}|_2^2
		\end{pmatrix},
	\end{split}
\end{align}
where $\delta_{-1},\delta_{0},\delta_{1}\in\R^n$ are written as columns vectors.
\end{defi}
\begin{rmk}
	For $n=2$, observe that for any $\delta_{-1},\delta_{0},\delta_{1}\in\R^2$, we have that
	\begin{align*}
		\text{num} &=
		\det\left(\begin{pmatrix}
			\delta_{-1} &  \delta_{1}
		\end{pmatrix}^T\begin{pmatrix}
			\delta_0 & \delta_{1}\
		\end{pmatrix}\right)
		=
		\det\begin{pmatrix}
			\delta_{-1} & \delta_{1}
		\end{pmatrix}\det\begin{pmatrix}
		\delta_{0} & \delta_{1}
	\end{pmatrix}, \\
	\text{den} &=
	\det\left(\begin{pmatrix}
		\delta_{-1} &  \delta_{1}
	\end{pmatrix}^T\begin{pmatrix}
		\delta_{-1} & \delta_{1}\
	\end{pmatrix}\right)
	=
	\det\begin{pmatrix}
		\delta_{-1} & \delta_{1}
	\end{pmatrix}^2.
\end{align*}
Hence, $A$ can be rewritten as
	\begin{equation}
		\label{eq:AB2}
			A(\delta_{-1},\delta_{0},\delta_{1}) :=    
			\begin{cases}
				\frac 12, & \text{if } \prod_{i\neq j}\det\begin{pmatrix}
					\delta_i & \delta_{j} 
				\end{pmatrix}= 0, \\
				\text{abs}\left (\frac{\det\begin{pmatrix}
				\delta_0 & \delta_{1} 
			\end{pmatrix}}{
			\det\begin{pmatrix}
				\delta_{-1} & \delta_{1} 
			\end{pmatrix}}\right ), & \text{otherwise}.
		\end{cases}
	\end{equation}
\end{rmk}

Note that $\alpha^k$ and $\beta^k$ involve truncating a number to ensure it always falls between two values dependent on $2^{-k}\rho$. This approach has been carefully designed to maintain the reproduction capability intact (as demonstrated in Theorem \ref{teo:scheme_reproduction}) while simultaneously ensuring the $\cC^1$ convergence of the scheme (as established in Corollary \ref{cor:scheme_C1}).
Observe that $\lim_{k\to+\infty} (\alpha^k,\beta^k) = (1,1)$, so that the proposed subdivision operators converge to the ones of the Lagrange scheme for uniform grids (see Remark \ref{rmk:dyn_scheme}). In particular, for $\rho = 0$, the proposed operator is exactly that one, which is linear and possesses good properties. With this, the asymptotic properties of our scheme, including convergence and regularity, are linked with the mentioned linear scheme.

Note that, for every $\rho>0$, a distinct scheme is obtained, which is non-stationary (since $\alpha^k$ and $\beta^k$ vary among iterations), uniform (only two subdivision rules per iteration are required, which do not consider any grid information), and non-linear. As shown in Section \ref{sec:numerical}, $\rho$ is related with the curve flexibility between points. We have observed that pleasing shapes emerge for $\rho$ values within the range of $[1,4]$, but this interval can be even wider, depending on the stylistic preferences of the end-user.

Observe that if $\alpha^k,\beta^k$ are defined without truncation, their image set is $(0,+\infty)$, which is unbounded and induces practical problems. They could be truncated with a constant number, independent of $k$, which would lead to a stationary scheme. But we did not find any benefit of considering such stationary version, which introduces several difficulties in proving convergence and reproduction simultaneously.

Before delving deeper into the analysis, it is essential to consider potential issues related to the scheme definition.

\begin{prop}
	For any $\rho\geq0$, the subdivision scheme in Definition \ref{defi:scheme} is well-defined.
\end{prop}
\begin{proof}
	For any $\delta_{-1},\delta_0,\delta_1\in\R^n$, we have that $A,B$ are well-defined and positive. Then, $\alpha^k$ and $\beta^k$, which are well-defined, can be composed with them and, by definition,
	\begin{equation}  \label{eq:alpha_bounded}
	((\alpha^k,\beta^k)\circ(A,B))(\delta_{-1},\delta_0,\delta_1)\in [(1+2^{-k}\rho)^{-1},1+2^{-k}\rho]^2.
\end{equation}
	The functions $a_{-1},a_0,a_1$ are being evaluated in positive numbers, so the denominators are positive and we conclude that $\Psi^k_0,\Psi^k_1$ are well-defined.
\end{proof}

The concepts presented in this paper can be applied to create similar schemes for different purposes. For example, an interpolatory subdivision scheme would be very practical. However, in our study, we opted to explore a `dual' scheme, which, in essence, generates new data more akin to the Chaikin's scheme. The reason for this choice is revealed in the following lemma, where one can observe that this particular refinement scheme offers a highly advantageous property. We leverage this property with the truncated-style definition of $\alpha^k$ and $\beta^k$, which becomes key to proving $\cC^1$ convergence and reproduction simultaneously.

\begin{lem} \label{lem:grid_contractivity}
	Let $\xi^{k},\xi^{k+1}$ be two grids fulfilling
	\begin{equation} \label{eq:chakin}
	\xi^{k+1}_{2i} = \frac34 \xi^k_{i} + \frac14 \xi^k_{i+1}, \qquad \xi^{k+1}_{2i+1} = \frac14 \xi^k_{i} + \frac34 \xi^k_{i+1}, \qquad \forall i\in\Z.
	\end{equation}
	Then, for any $\rho \geq 0$,
	\[
	(1+\rho)^{-1} \leq \alpha_i^k \leq 1+\rho, \ \forall i\in\Z \quad \longrightarrow \quad \left(1+\frac\rho2\right)^{-1} \leq \alpha_i^{k+1} \leq 1+\frac\rho2, \ \forall i\in\Z.
	\]
	\begin{proof}
		We are going to prove that $\alpha_{2i}^{k+1}=\frac{\nabla \xi^{k+1}_{2i+1}}{\nabla \xi^{k+1}_{2i}}$ and $\alpha_{2i+1}^{k+1}=\frac{\nabla \xi^{k+1}_{2i+2}}{\nabla \xi^{k+1}_{2i+1}}$ are in between these bounds, $\forall i\in\Z$.
		First, observe that
		\begin{equation} \label{eq:lem:grid_contractivity:1}
			(1+\rho)^{-1} \leq x \leq 1+\rho, \quad \longrightarrow \quad (1+\rho)^{-1} \leq x^{-1} \leq 1+\rho.
		\end{equation}
		Second, we deduce that
		\[
		\nabla \xi^{k+1}_{2i+1} = \xi^{k+1}_{2i+2}-\xi^{k+1}_{2i+1} = \frac14\xi^{k}_{i+2}-\frac14\xi^{k}_{i}, \quad
		\nabla \xi^{k+1}_{2i} = \xi^{k+1}_{2i+1}-\xi^{k+1}_{2i} = \frac12\xi^{k}_{i+1}-\frac12\xi^{k}_{i}.
		\]
		Then,
	\begin{align*}
		\frac{\nabla \xi^{k+1}_{2i+1}}{\nabla \xi^{k+1}_{2i}}-1 &= \frac{\frac14\xi^{k}_{i+2}-\frac14\xi^{k}_{i}}{\frac12\xi^{k}_{i+1}-\frac12\xi^{k}_{i}}-1
		= \frac12\left(\frac{\xi^{k}_{i+2}-\xi^{k}_{i}}{\xi^{k}_{i+1}-\xi^{k}_{i}}-2\right)
		= \frac12\left(\frac{\xi^{k}_{i+2}-\xi^{k}_{i}-\xi^{k}_{i+1}+\xi^{k}_{i}}{\xi^{k}_{i+1}-\xi^{k}_{i}}-1\right) \\
		&= \frac12\left(\frac{\nabla \xi^{k}_{i+1}}{\nabla \xi^{k}_{i}}-1\right),\\
		\frac{\nabla \xi^{k+1}_{2i+1}}{\nabla \xi^{k+1}_{2i+2}}-1 &= \frac{\frac14\xi^{k}_{i+2}-\frac14\xi^{k}_{i}}{\frac12\xi^{k}_{i+2}-\frac12\xi^{k}_{i+1}}-1
		= \frac12\left(\frac{\xi^{k}_{i+2}-\xi^{k}_{i}}{\xi^{k}_{i+2}-\xi^{k}_{i+1}}-2\right) = \frac12\left(\frac{\xi^{k}_{i+2}-\xi^{k}_{i}-\xi^{k}_{i+2}+\xi^{k}_{i+1}}{\xi^{k}_{i+2}-\xi^{k}_{i+1}}-1\right) \\
		&= \frac12\left(\frac{\nabla \xi^{k}_{i}}{\nabla \xi^{k}_{i+1}}-1\right).
	\end{align*}
Applying to the last two formulas the fact that
\[
x\in [(1+\rho)^{-1}, 1+\rho] \quad \lra \quad 1+\frac12(x-1)\in \left [\frac{2+\rho}{2+2\rho}, 1+\frac\rho2\right ] \subset \left [\left (1+\frac\rho2\right )^{-1}, 1+\frac\rho2\right ]
\]
(the latter can be checked in the Mathematica file, see Section Reproducibility), we arrive to
\[
\left (1+\frac\rho2\right )^{-1} \leq \frac{\nabla \xi^{k+1}_{2i+1}}{\nabla \xi^{k+1}_{2i}},\frac{\nabla \xi^{k+1}_{2i+1}}{\nabla \xi^{k+1}_{2i+2}} \leq 1+\frac\rho2.
\]
Using \eqref{eq:lem:grid_contractivity:1}, we arrive to the claim.
\end{proof}
\end{lem}

We are now prepared to demonstrate that the scheme reproduces non-degenerate parabolas (i.e., the image of $F\in\Pi_2^n$ is not confined to a line). Additionally, we establish that the scheme generates lines when the data is collinear.
\begin{teo} \label{teo:scheme_reproduction}
	Let us consider \(\mathcal{L} = \{F\in\Pi_2^n : F(\R) \text{ is contained in a line}\}\).
	For any $\rho\geq0$, the subdivision scheme in Definition \ref{defi:scheme} reproduces
	\(\cF = \Pi_2^n \setminus \mathcal{L}\)
	on any $\xi^0$ grid fulfilling
	\[
	(1+\rho)^{-1}\leq \alpha^0_{i} \leq 1+\rho, \quad i\in\Z.
	\]
	Moreover, if $f^0$ lies on a line, then $f^k$ lies on the same line.
\end{teo}
\begin{proof}
	Using Proposition \ref{prop:step_implies_limit}, it is sufficient to prove the stepwise reproduction for some sequence of grids $\{\xi^{k}\}_{k\in\N_0}$. We define the grids using the Chaikin's subdivision scheme, as in \eqref{eq:chakin}.
	
	Denote $f^k = F|_{\xi^k}\in\liZn$, $F\in\cF$. Let us focus on four consecutive data $f^k_{i-1},f^k_{i},f^k_{i+1},f^k_{i+2}$, that indeed fulfil
	\begin{equation} \label{eq:hypothesis}
		F(\xi^k_j)=f^k_{j}, \quad j=i-1,\ldots,i+2.
	\end{equation}
	We have to prove that
	\[
	\Psi^k_l(f^k_{i-1},f^k_{i},f^k_{i+1},f^k_{i+2}) = F(\xi^{k+1}_{2i+l}), \quad l=0,1.
	\]
	
	Considering that the proposed scheme is primarily built upon the Lagrange scheme of Definition \ref{defi:lagrange}, which stepwise reproduces $\Pi_2^n$ according to Proposition \ref{prop:lagrange_reproduces}, our focus shifts to confirming that the non-linear scheme exactly computes $\alpha^k_i$ and $\beta^k_i$ from the data. This requires us to verify that
	\[
	(\alpha^k_i,\beta^k_i) = ((\alpha^k,\beta^k)\circ(A,B))(\nabla f^k_{i-1},\nabla f^k_{i},\nabla f^k_{i+1}),
	\]
	because it would imply
	\[
		\Psi^k_l(f^k_{i-1},f^k_{i},f^k_{i+1},f^k_{i+2}) = \Lambda_{l}[\alpha^k_{i},\beta^k_{i}](f^k_{i-1},f^k_{i},f^k_{i+1},f^k_{i+2}) = (L_k f^k)_{2i+l} = F(\xi^{k+1}_{2i+l}), \qquad l=0,1.
	\]
	
	Recall the annihilation property for $\Pi^2_n$ in Proposition \ref{prop:annihilation}:
	\begin{equation} \label{eq:annihilation}
		A_i^k \nabla f^k_{i-1} - \nabla f^k_i + B_i^k \nabla f^k_{i+1} = 0\in\R^n,
	\end{equation}
	where
	\begin{equation} \label{eq:AB_from_alphabeta}
		A_i^k=\frac{\beta_i^k+1}{\alpha_i^k (\alpha_i^k+\beta_i^k+2)}, \qquad B_i^k=\frac{\alpha_i^k+1}{\beta_i^k (\alpha_i^k+\beta_i^k+2)}.
	\end{equation}
	To ensure that $\alpha^k_i,\beta^k_i$ can be recovered from $A^k_i,B^k_i$, we observe that the function
	\[(\alpha,\beta)\mapsto\left(\frac{\beta+1}{\alpha (\alpha+\beta+2)},\frac{\alpha+1}{\beta (\alpha+\beta+2)}\right)\]
	is bijective from $(0,+\infty)^2$ to itself, with inverse
	\[
	(A,B)\mapsto\left(\frac{1}{A+\frac{\sqrt{A (A+1) B (B+1)}}{B+1}},\frac{1}{B+\frac{\sqrt{A (A+1) B (B+1)}}{A+1}}\right), \quad A,B>0.\]
	See the Mathematica file (Section Reproducibility) for a verification.
	Hence, $\alpha^k_i = \alpha^k(A^k_i,B^k_i)$ and $\beta^k_i = \beta^k(A^k_i,B^k_i)$, with $\alpha^k,\beta^k$ as defined in \eqref{eq:alpha_beta_k}, provided that \[(1+2^{-k}\rho)^{-1}\leq \alpha^k_i,\beta^k_i \leq 1+2^{-k}\rho.\]
	We can ensure the last condition for $k=0$, by hypothesis, and for any $k>0$, by applying Lemma \ref{lem:grid_contractivity} recursively.
	
	The last point of this proof consists in checking that
	\begin{equation} \label{eq:AkAcoincides}
	A^k_i = A(\nabla f^k_{i-1},\nabla f^k_{i},\nabla f^k_{i+1}), \quad
	B^k_i = B(\nabla f^k_{i-1},\nabla f^k_{i},\nabla f^k_{i+1}),
\end{equation}
with $A,B$ as defined in \eqref{eq:AB}.
	It is done by solving the $n\times 2$ linear system \eqref{eq:annihilation}, with unknowns $A_i^k,B_i^k$. Despite the system size, we know that there exists a solution, the one in \eqref{eq:AB_from_alphabeta} and it is unique provided that $\text{rank}\begin{pmatrix}
		\nabla f^k_{i-1} & \nabla f^k_{i+1}
	\end{pmatrix} = 2$. In that case, the solution can be written as:
	\[
	\begin{pmatrix}
		A^k_i \\ B^k_i
	\end{pmatrix}
	=
	\left (
	\begin{pmatrix}
		\nabla f^k_{i-1} & \nabla f^k_{i+1} 
	\end{pmatrix}^T
	\begin{pmatrix}
		\nabla f^k_{i-1} & \nabla f^k_{i+1} 
	\end{pmatrix}\right )^{-1}
	\begin{pmatrix}
		\nabla f^k_{i-1} & \nabla f^k_{i+1} 
	\end{pmatrix}^T
	\nabla f^k_i,
	\]
	where we are writing $\nabla f^k_i\in\R^n$ as column vectors. Performing the matrices multiplications:
	\[
	\begin{pmatrix}
		A^k_i \\ B^k_i
	\end{pmatrix}
	=
	\begin{pmatrix}
		|\nabla f^k_{i-1}|_2^2 & (\nabla f^k_{i-1})^T \nabla f^k_{i+1}\\
		(\nabla f^k_{i-1})^T \nabla f^k_{i+1} & |\nabla f^k_{i+1}|_2^2
	\end{pmatrix}^{-1}
	\begin{pmatrix}
		(\nabla f^k_{i-1})^T \nabla f^k_i \\ (\nabla f^k_{i+1})^T \nabla f^k_i
	\end{pmatrix}.
	\]
	Applying the Cramer's formula, we arrive to the definition of $A,B$ in \eqref{eq:AB}, provided that $\text{num}\cdot \text{den} \neq 0$. Since $\text{rank}\begin{pmatrix}
		\nabla f^k_{i-1} & \nabla f^k_{i+1}
	\end{pmatrix} = 2$, then $\text{den} \neq 0$.
	By \eqref{eq:AB_from_alphabeta}, we know that $A^k_i,B^k_i$ are positive numbers, implying that $\text{num} \neq 0$.

Now we address the case $\text{rank}\begin{pmatrix}
	\nabla f^k_{i-1} & \nabla f^k_{i+1}
\end{pmatrix} < 2$, that is, when $\nabla f^k_{i-1}$ and $\nabla f^k_{i+1}$ are proportional. 
According to \eqref{eq:annihilation},
\begin{equation} \label{eq:in_plane}
\nabla f^k_i = A_i^k \nabla f^k_{i-1}  + B_i^k \nabla f^k_{i+1}.
\end{equation}
Thus, we have that the three vectors are proportional, meaning that $f^k_{i-1},f^k_{i},f^k_{i+1},f^k_{i+2}$ are co-linear. This implies that $F\in\mathcal{L}$. Since $\Psi^k_l$, $l=0,1$, consists of applying $\Lambda_l[\alpha,\beta]$ for some $\alpha,\beta$ and, by Remark \ref{rmk:lagrange_affine}, it always performs convex combinations of the data, then $f^{k+1}_{2i},f^{k+1}_{2i+1}$ lies on the same line as $f^k_{i-1},f^k_{i},f^k_{i+1},f^k_{i+2}$.
\end{proof}

\begin{rmk}
	According to \eqref{eq:in_plane}, if $f^k = F|_{\xi^k}$, $F\in\Pi_2^n$, then the three vectors $\nabla f^k_{i-1},\nabla f^k_{i},\nabla f^k_{i+1}$ are linearly dependent and, as a result, $f^k_{i-1}, f^k_{i}, f^k_{i+1}, f^k_{i+2}$ lie on the same plane.
	If $F\in \Pi_2^n \setminus \mathcal{L}$, then $\nabla f^k_{i-1},\nabla f^k_{i+1}$ are linearly independent and they form a basis of that plane. In addition, since $A_i^k,B_i^k>0$, then $\nabla f^k_i$ falls into the first quadrant defined by this basis. This is illustrated in Figure \ref{fig:quadrants}-(a). We can also see that $\nabla f^k_i$ cannot fall into any other quadrant when the data describe a parabola.

	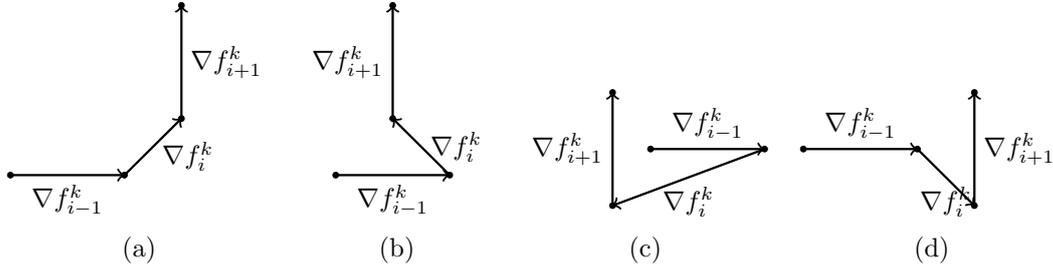
\begin{figure}[h]
		\centering
		\begin{tabular}{cccc}
			\begin{tikzpicture}[scale=1.5]
				\draw[->, line width=0.3mm] (-1,0) -- (0,0);
				\node[below] at (-1/2,0) {$\nabla f^k_{i-1}$};
				\draw[->, line width=0.3mm] (0,0) -- (1/2,1/2);
				\node[right] at (1/4,1/6) {$\nabla f^k_{i}$};
				\draw[->, line width=0.3mm] (1/2,1/2) -- (1/2,3/2) node[above] {};
				\node[right] at (1/2,1) {$\nabla f^k_{i+1}$};
				\draw[fill] (-1,0) circle [radius=0.025];
				\draw[fill] (0,0) circle [radius=0.025];
				\draw[fill] (1/2,1/2) circle [radius=0.025];
				\draw[fill] (1/2,3/2) circle [radius=0.025];;
			\end{tikzpicture}
			&
			\begin{tikzpicture}[scale=1.5]
				\draw[->, line width=0.3mm] (-1,0) -- (0,0);
				\node[below] at (-1/2,0) {$\nabla f^k_{i-1}$};
				\draw[->, line width=0.3mm] (0,0) -- (-1/2,1/2);
				\node[right] at (-1/4,1/4) {$\nabla f^k_{i}$};
				\draw[->, line width=0.3mm] (-1/2,1/2) -- (-1/2,3/2) node[above] {};
				\node[left] at (-1/2,1) {$\nabla f^k_{i+1}$};
				\draw[fill] (-1,0) circle [radius=0.025];
				\draw[fill] (0,0) circle [radius=0.025];
				\draw[fill] (-1/2,1/2) circle [radius=0.025];
				\draw[fill] (-1/2,3/2) circle [radius=0.025];
			\end{tikzpicture}
			&
			\begin{tikzpicture}[scale=1.5]
				\draw[->, line width=0.3mm] (-1,0) -- (0,0);
				\node[above] at (-1/2,0) {$\nabla f^k_{i-1}$};
				\draw[->, line width=0.3mm] (0,0) -- (-4/3,-1/2);
				\node[below] at (-4/6,-1/4) {$\nabla f^k_{i}$};
				\draw[->, line width=0.3mm] (-4/3,-1/2) -- (-4/3,1/2) node[above] {};
				\node[left] at (-4/3,0) {$\nabla f^k_{i+1}$};
				\draw[fill] (-1,0) circle [radius=0.025];
				\draw[fill] (0,0) circle [radius=0.025];
				\draw[fill] (-4/3,-1/2) circle [radius=0.025];
				\draw[fill] (-4/3,1/2) circle [radius=0.025];
			\end{tikzpicture}
			&
			\begin{tikzpicture}[scale=1.5]
				\draw[->, line width=0.3mm] (-1,0) -- (0,0);
				\node[above] at (-1/2,0) {$\nabla f^k_{i-1}$};
				\draw[->, line width=0.3mm] (0,0) -- (1/2,-1/2);
				\node[below] at (1/4,-1/4) {$\nabla f^k_{i}$};
				\draw[->, line width=0.3mm] (1/2,-1/2) -- (1/2,1/2) node[above] {};
				\node[right] at (1/2,0) {$\nabla f^k_{i+1}$};
				\draw[fill] (-1,0) circle [radius=0.025];
				\draw[fill] (0,0) circle [radius=0.025];
				\draw[fill] (1/2,-1/2) circle [radius=0.025];
				\draw[fill] (1/2,1/2) circle [radius=0.025];
			\end{tikzpicture}
			\\
			(a) & (b) & (c) & (d)
		\end{tabular}
		\caption{Four cases where $\nabla f^k_{i-1},\nabla f^k_{i+1}$ are linearly independent and $\nabla f^k_i$ falls into the first, second, third or fourth quadrant (from the left to the right).}
		\label{fig:quadrants}
	\end{figure}

	If $\nabla f^k_{i-1},\nabla f^k_{i+1}$ are linearly dependent, then $\nabla f^k_{i}$ is also proportional to them by \eqref{eq:in_plane}, and the four data points are collinear. This happens when $F\in\mathcal{L}$, for instance, $F(x) = (x^2,x^2)$.
We would also like to emphasize that if $\nabla f^k_i = 0$, then $\nabla f^k_{i-1},\nabla f^k_{i+1}$ are linearly dependent with opposite directions according to \eqref{eq:in_plane}.

It becomes impossible to determine $F(\xi^{k+1}_{2i})$ and $F(\xi^{k+1}_{2i+1})$ in the last case:
Since $\nabla f^k_{i-1},\nabla f^k_{i+1}$ are linearly dependent, the rank of the system \eqref{eq:annihilation} is one, and only one equation (one coordinate) is sufficient to represent the system information. This is insufficient for reproducing the data,
and neither this scheme nor any other one can reproduce the data without knowing $\alpha^k_i,\beta^k_i$, as we demonstrate in the following example: Let us consider the family of polynomials
\[
\left \{p(x) = \frac{x(x-1)}{\alpha(\alpha+1)} : \alpha>0\right \} \subset \Pi_2^1.
\]
All of them fulfil that
\[
p(-\alpha) = 1, \quad p(0) = 0, \quad p(1) = 0, \quad p(\alpha) = 1,
\]
for some $\alpha>0$.
Consequently, provided the four data points $1, 0, 0, 1,$ it becomes impossible to determine $\alpha$ and, by extension, $p(\frac14)$ and $p(\frac34)$. For an example in $\Pi_2^n$, $n\geq 2$, just consider the proposed family of polynomials in each coordinate.
\end{rmk}

\begin{rmk}
	As a consequence of Theorem \ref{teo:scheme_reproduction}, we can prove that the scheme is not $\cC^2$, for any $\rho>0$. For $\rho=0$, it is the linear scheme proposed in \cite{DFH04} which is $\cC^2$.
	Let us consider the initial grid $\xi^0_i = (1+\rho)^i$, $i\in\Z$, so that
	\[
	\alpha^0_i = \frac{\nabla \xi^0_{i+1}}{\nabla \xi^0_{i}} = 1+\rho.
	\]
	Let us consider the initial data $f^0 = F|_{\xi^0}$ with $F(x) = (x,x^2)$.
	
	On the one hand, by Theorem \ref{teo:scheme_reproduction} and Proposition \ref{prop:step_implies_limit}, $S^\infty(f^0) = F\circ\sigma$, where $\sigma$ is the limit function of the Chaikin's scheme applied to $\xi^0$, which is $\cC^{2-}$. More precisely, $\sigma$ is a quadratic spline whose second derivative is piecewise constant, with jumps at the integers. Thus, $S^\infty(f^0) = F\circ\sigma = (\sigma,\sigma^2)$ is also $\mathcal{C}^{2-}$.
	
	On the other hand, the limit curve is precisely a parabola, which is infinitely smooth. This example underscores that the perception of `smoothness' depends on whether we treat the limit as a function or a curve. It emphasizes that a smooth curve does not necessarily imply a smooth function (its parametrization). Conversely, the cycloidal curve serves as a classical illustration which demonstrates that a smooth parametrization does not guarantee a smooth curve.
	
	We show in the numerical experiments of Section \ref{sec:numerical} that the scheme generates curves with continuous curvature. We hope that this could be theoretically established in the future. For the moment, we prove in Section \ref{sec:convergence_asymptotic} that it is $\cC^1$ for any initial data and any choice of $\rho\geq 0$.
	
	This example also demonstrates that the regularity of the scheme is constrained by the regularity of the implicit mesh refinement. So, replacing the Chaikin's scheme with some other smoother and monotonicity preserving scheme may be advantageous.
\end{rmk}

In the following proposition, we establish that the scheme commutes with similarity transformations, including rotations, translations, scalings, and reflections. This property is desirable for CAD applications, ensuring consistent outcomes when applying these standard transformations to the data. It is noteworthy that any similarity transformation is an affine transformation, so the next result is weaker than the one in Remark \ref{rmk:lagrange_affine}.

\begin{prop} \label{prop:similarity_transformations}
	For any $\rho\geq0$, the subdivision scheme in Definition \ref{defi:scheme} commutes with similarity transformations.
\end{prop}
\begin{proof}
	Similarly to Remark \ref{rmk:lagrange_affine}, we need to prove that if a similarity transformation is applied to the data, i.e.
	\begin{equation} \label{eq:prop:similarity_transformations}
		\bar f_i := \mu M f_i + c, \qquad \forall i=-1,0,1,2,
	\end{equation}
	being $\mu\in\R$, $c\in\R^n$ and $M\in\R^{n\times n}$ an orthogonal matrix, then
	\begin{equation} \label{eq:prop:similarity_transformations:rule}
		\Psi_l^k(\bar f_{-1},\bar f_{0},\bar f_{1},\bar f_{2}) = \mu M \Psi^k_l(f_{-1},f_{0},f_{1},f_{2}) + c, \quad l=0,1.
	\end{equation}

	The case $\mu = 0$ is trivial since $\Lambda_l[\alpha,\beta]$ reproduces constants for any $\alpha,\beta>0$, $l=0,1$.
	For $\mu\neq 0$, first observe that
	\[ \nabla \bar f_i = \mu M \nabla f_i, \qquad \forall i=-1,0,1,2. \]
	Then, looking at \eqref{eq:AB}, it is straightforwardly verified that
	\[
		(A,B)(\nabla f_{-1},\nabla f_{0},\nabla f_{1}) = (A,B)(\nabla \bar f_{-1},\nabla \bar f_{0},\nabla \bar f_{1}).
	\]
	Consequently,
	\[\Lambda_{l}[((\alpha^k,\beta^k)\circ(A,B))(\nabla f_{-1},\nabla f_{0},\nabla f_{1})]
	= \Lambda_{l}[((\alpha^k,\beta^k)\circ(A,B))(\nabla \bar f_{-1},\nabla \bar f_{0},\nabla \bar f_{1})], \quad l=0,1.\]
	Then, \eqref{eq:prop:similarity_transformations:rule} is achieved provided that $\Lambda_l[\alpha,\beta]$, $l=0,1$, also fulfil \eqref{eq:prop:similarity_transformations:rule} for any $\alpha,\beta>0$, which certainly holds true by Remark \ref{rmk:lagrange_affine}.
\end{proof}

\section{Convergence by asymptotic equivalence and quasilinearity}	\label{sec:convergence_asymptotic}

In this section, we introduce a broad class of non-linear, non-stationary, uniform subdivision schemes, which includes the proposed scheme. We demonstrate that these schemes are $\mathcal{C}^1$ by utilizing two concepts: their quasilinearity (refer to \cite{CDM03}) and their asymptotic equivalence (refer to \cite{DLY14}) to linear, stationary, uniform schemes.
To apply the findings from these papers, we assume that the data is bounded, i.e., $f^0\in\liZn$, $S_k : \liZn \lra \liZn$, $k\in\N_0$. This is a common constraint in subdivision theory, which is readily met in practice due to the finite nature of the data.

Let us introduce each of these concepts.

\begin{defi}[Quasilinear subdivision scheme. Based on Definition 1 of \cite{CDM03}]
	A subdivision operator $S$ is \emph{data dependent} if for each $f\in\liZn$ there exists a linear subdivision operator $S[f]$ such that
	\(
	S f = S[f] f.
	\)
	A subdivision scheme is \emph{quasilinear} if each subdivision operator is data dependent.
\end{defi}
\begin{rmk}
	According to the Euler's theorem, any function $F:\R^n\lra\R$ that is 1-homogeneous and continuously differentiable fulfils \(F(x) = \langle\nabla F(x),x\rangle \). This can be applied to the subdivision rules to deduce that any 1-homogeneous subdivision scheme that admits generalized Jacobian (see \cite{DLS17} for further information) is quasilinear. Thus, quasilinearity is a very common property among the non-linear subdivision schemes in the literature.
\end{rmk}

\begin{rmk} \label{rmk:quasilinear_notation}
	Let $\cS = \{S_k\}_{k\in\N_0}$ be a quasilinear subdivision scheme. For a given $f^0\in\liZn$, consider $f^{k+1} = S_k f^k$, $k\in\N_0$. $S_k[f^k]$ denotes the linear subdivision operator such that $f^{k+1} = S_k[f^k] f^k$. We consider the linear (and possibly non-uniform) subdivision scheme $\cS[f^0] := \{S_k[f^k]\}_{k\in\N_{0}}$. Observe that, for each $f^0\in\liZn$, $\cS$ and $\cS[f^0]$ generate exactly the same sequences $\{f^k\}_{k\in\N_0}$. We will exploit this fact to examine the smoothness of $\cS$ through $\cS[f^0]$. For this analysis, we can apply the results for linear schemes, particularly those found in \cite{DLY14}.
\end{rmk}

\begin{prop} \label{prop:quasilinearCm}
	Let $\cS = \{S_k\}_{k\in\N_0}$ be a quasilinear subdivision scheme and let $m\in\N_0$ be. If $\cS[f^0]$ is $\cC^m$, $\forall f^0\in\liZn$, then $\cS$ is $\cC^m$.
\end{prop}
\begin{proof}
	For each $f^0\in\liZn$, the subdivision schemes $\cS$ and $\cS[f^0]$ produce exactly the same sequence $\{f^k\}_{k\in\N_{0}}$. Since $\cS[f^0]$ converges to a $\cC^m$ function, $\cS$ also converges to that $\cC^m$ function when starting from $f^0$.
\end{proof}

We now revisit the property of asymptotic equivalence, which was previously examined for linear, non-uniform, non-stationary subdivision schemes in \cite{DLY14}. This property has proven to be effective in the analysis of subdivision schemes that reproduce extended Chebyshev systems. The findings presented in that paper are exclusively concerned with linear scalar-valued subdivision schemes, i.e., $n=1$. However, by applying these results to each component individually, they also hold true for linear vector-valued scalar-based subdivision schemes.

Similarly to \cite{DLY14}, given a linear scalar-based subdivision scheme $\cS = \{S_k\}_{k\in\N_0}$, we denote by $b^{i,k}\in\ell(\Z,\R)$, $i\in\Z$, $k\in\N_0$, some \emph{masks} that allow to compute the refinement as follows:
\[
(S_{k} f^{k})_i = \sum_{j\in\Z} b^{i,k}_{i-2j} f^{k}_j, \quad i\in\Z, \ f^k\in\liZn.
\]
Observe that only half of the coefficients of each mask are used (the even or the odd ones, depending on the parity of $i$), so that the unused half of $b^{i,k}$ can be defined in several ways, for instance, in some convenient manner for a specific purpose (as in Theorem \ref{teo:nonlinear_asymtotic_equivalent}).

The next two properties link the convergence and regularity of two linear subdivision schemes.
\begin{defi}[Asymptotically equivalent. Based on Section 3.2 of \cite{DLY14}]
	Let $\cS = \{S_k\}_{k\in\N_0},\tilde \cS = \{\tilde S_k\}_{k\in\N_0}$ be two linear scalar-based subdivision schemes. They are \emph{asymptotically equivalent} if
	\[
	\sum_{k=0}^\infty \sup_{i\in\Z} \{\|b^{i,k} - {\tilde b}^{i,k}\|\} < \infty,
	\]
	where $\|\cdot\|$ is any norm for the masks.
\end{defi}

\begin{defi}[Property A. Based on Definition 4.1 of \cite{DLY14}]
	Consider the Laurent polynomials
	\[
	b^{i,k}(z) := \sum_{j\in\Z} b^{i,k}_j z^j, \quad d^{i,k}_m(z) := \sum_{j=0}^m (-1)^j \binom{m}{j} z^j b^{i-j,k}(z), \quad m\in\N, \ i\in\Z, \ k\in\N_0.
	\]
	A linear scalar-based scheme $\cS$ satisfies \emph{property A of order $m$} if
	\[
	\sum_{k=0}^\infty 2^{k(m-r)}\left | \frac{\partial^r d^{i,k}_m}{\partial z^r} (\pm 1) \right | < \infty, \qquad 0\leq r< m, \quad i\in\Z.
	\]
\end{defi}

Now, we propose a straightforward result that links Property A with the reproduction of constants.
\begin{lem} \label{lem:sum_one_implies_d0}
If $\sum_{j\in\Z}b^{i,k}_{2j}=\sum_{j\in\Z}b^{i,k}_{1+2j} = 1$, $\forall i\in\Z$, then $d^{i,k}_m(\pm 1) = 0$, $\forall i\in\Z$, $\forall m\in\N$.
\end{lem}
\begin{proof}
The hypothesis implies that the symbols of the masks fulfil that $b^{i,k}(1) = 2$ and $b^{i,k}(-1) = 0$. Then
\begin{align*}
	d^{i,k}_m(1) &= \sum_{j=0}^m (-1)^j \binom{m}{j}\cdot 1\cdot 2 = 0, \quad
	d^{i,k}_m(-1) = \sum_{j=0}^m (-1)^j \binom{m}{j} (-1)^j\cdot 0 = 0.
\end{align*}
\end{proof}

\begin{teo}[Proposition 3.2 of \cite{DLY14} and the latter comment] \label{teo:Am_implies_C0}
	Let $\cS,\tilde \cS$ be two asymptotically equivalent subdivision schemes, being $\tilde \cS$ uniform and stationary. If $\tilde \cS$ converges, then $\cS$ does.
\end{teo}

\begin{teo}[Theorem 4.2 of \cite{DLY14}] \label{teo:Am_implies_Cm}
Let $\cS,\tilde \cS$ be two asymptotically equivalent subdivision schemes, being $\tilde \cS$ uniform, stationary and $\cC^m$, for some $m\in\N$. If $\cS$ satisfies property A of order $m$, then $\cS$ is $\cC^m$.
\end{teo}

Now, we show how to establish a connection between the concepts of quasilinearity and asymptotic equivalence.

\begin{cor} \label{cor:quasilinear_propertyA_implies_Cm}
Let $\cS$ be a quasilinear subdivision scheme. Let $\tilde \cS$ be a convergent, linear, uniform, stationary and scalar-based subdivision scheme. 
If $\cS[f^0]$ is scalar-based and asymptotically equivalent to $\tilde \cS$, $\forall f^0\in\liZn$, then $\cS$ is convergent.
In addition, if $\cS[f^0]$ satisfies property $A$ of order $m$, $\forall f^0\in\liZn$, for some $m\in\N$, and $\tilde \cS$ is $\cC^m$, then $\cS$ is $\cC^m$.
\end{cor}
\begin{proof}
	First, apply Theorem \ref{teo:Am_implies_C0} (and Theorem \ref{teo:Am_implies_Cm}) to obtain that $\cS[f^0]$ is convergent (and $\cC^m$), $\forall f^0\in\liZn$. Second, apply Proposition \ref{prop:quasilinearCm} to deduce that $\cS$ is convergent (and $\cC^m$) as well.
\end{proof}

Next, we propose a broad class of subdivision schemes, ours included, and apply Corollary \ref{cor:quasilinear_propertyA_implies_Cm} to achieve the convergence of any scheme within this class. We believe this could be useful in the study of future non-linear non-stationary subdivision schemes.

\begin{defi}[Class B] \label{defi:classB}
	A subdivision scheme $\cS$ belongs to the \emph{class B} iff it is uniform and there exist $q\in\N$, $D\subset \R^q$, $0\in D$, $\pi_k:(\R^n)^{2p+2}\lra D$, $\vec\Lambda_{0},\vec\Lambda_{1}: D \lra \R^{2p+2}$, $\rho,C> 0$ such that
	\begin{align*}
			\Psi^k_l(f_{-p},\ldots,f_{p+1}) &=  \langle \vec\Lambda_{l}(\pi_k(f_{-p},\ldots,f_{p+1})),(f_{-p},\ldots,f_{p+1})\rangle, & l=0,1, \  \forall f_j\in\R^n, \ j=-p,\ldots,i+1,\\
			|\pi_k(f_{-p},\ldots,f_{p+1})|_{\R^q} &\leq 2^{-k} \rho, & \forall f_j\in\R^n, \ j=-p,\ldots,i+1,\\
			|\vec\Lambda_{l}(\pi) - \vec\Lambda_{l}(0)|_{\R^{2p+2}} &\leq  C|\pi|_{\R^q}, & \forall\pi\in D, \  l=0,1,
	\end{align*}
	for some vector norms $|\cdot|_{\R^q},|\cdot|_{\R^{2p+2}}$. We say that $\cS$ is \emph{asymptotically equivalent} to a linear uniform stationary subdivision scheme, that we denote $\cS^*$, whose rules are given by the coefficients $\vec\Lambda_{l}(0)$, $l=0,1$:
	\[
		\Psi^*_l(f_{-p},\ldots,f_{p+1}) =  \langle \vec\Lambda_{l}(0),(f_{-p},\ldots,f_{p+1})\rangle.
	\]
	We demand that $\cS^*$ is convergent.
\end{defi}

Next theorem justifies that the term \emph{asymptotically equivalent} in last definition is suitable.
\begin{teo} \label{teo:nonlinear_asymtotic_equivalent}
	If $\cS$ belongs to class B, then
	\begin{enumerate}
		\item $\cS$ is quasilinear
		\item The schemes $\cS[f^0]$ and $\cS^*$ are scalar-based, $\forall f^0 \in \liZn$.
		\item For a given $f^k\in\liZn$, the masks of $S_k[f^k]$, say $\{b^{i,k}\}_{i\in\Z}$, can be chosen as
		\begin{equation} \label{eq:mask_classB}
		(b^{2i,k}_{l-2j})_{j=-p}^{p+1} = \vec\Lambda_{l}(\pi_k(f^k_{i-p},\ldots,f^k_{i+p+1})), \qquad l=0,1,
		\end{equation}
		with $b^{2i,k}_{j} = 0$ if $j<-2p-2$ or $j>1+2p$,
		and $b^{2i+1,k} = b^{2i,k}$, $\forall i\in\Z$.
		\item $\cS[f^0]$ is asymptotically equivalent to $\cS^*$, $\forall f^0 \in \liZn$.
	\end{enumerate}
\end{teo}
\begin{proof}
It is quasilinear since $S_k f^k = S_k[f^k] f^k$ where $S_k[f^k]$ is precisely
\begin{equation} \label{eq:quasilinear_rule}
(S_k[f^k]g)_{2i+l} = \langle\vec\Lambda_{l}(\pi_k(f^k_{i-p},\ldots,f^k_{i+p+1})),(g_{i-p},\ldots,g_{i+p+1})\rangle, \quad l=0,1, \ g\in\liZn.
\end{equation}
Observe that the subdivision rules of $\cS[f^0]$ and $\cS^*$ consists of linear combinations (with scalar coefficients) of vector data. Thus, they are scalar-based.

Let us consider an arbitrary but fixed $f^0\in\liZn$. Using the notation in Remark \ref{rmk:quasilinear_notation}, we need to define $b^{i,k}$, $i\in\Z$, such that
\[
\langle\vec\Lambda_{l}(\pi_k(f^{k}_{i-p},\ldots,f^{k}_{i+p+1})),(g_{i-p},\ldots,g_{i+p+1})\rangle = (S_{k}[f^{k}] g)_{2i+l} = \sum_{j\in\Z} b^{2i+l,k}_{l+2(i-j)} g_j =  \sum_{j\in\Z} b^{2i+l,k}_{l-2j} g_{i+j}, \quad l=0,1.
\]
Hence, 
\[
	(b^{2i+l,k}_{l-2j})_{j=-p}^{p+1} := \vec\Lambda_{l}(\pi_k(f^k_{i-p},\ldots,f^k_{i+p+1})).
\]
We are omitting the dependence of $b^{2i+l,k}$ on $f^{k}$. For $l=0$, this is
\[
b^{2i,k}_{-2j} = \vec\Lambda_{0}(\pi_k(f^k_{i-p},\ldots,f^k_{i+p+1})), \quad j=-p,\ldots,p+1.
\]
Since the odd positions of $b^{2j,k}$ are not used, we can freely define them as
\[
b^{2i,k}_{1-2j} := \vec\Lambda_{1}(\pi_k(f^k_{i-p},\ldots,f^k_{i+p+1})), \quad j=-p,\ldots,p+1.
\]
We can proceed analogously with $l=1$, obtaining that
\[
b^{2i+1,k}_{1-2j} = \vec\Lambda_{1}(\pi_k(f^k_{i-p},\ldots,f^k_{i+p+1})), \quad j=-p,\ldots,p+1,
\]
and defining
\[
b^{2i+1,k}_{-2j} := \vec\Lambda_{0}(\pi_k(f^k_{i-p},\ldots,f^k_{i+p+1})), \quad j=-p,\ldots,p+1.
\]
By defining the unused coefficients in this manner, we have got \eqref{eq:mask_classB} and 
\(
b^{2i,k} = b^{2i+1,k}, \ \forall i\in\Z.
\)

Denote by $b^*$ the mask of the linear, uniform and stationary scheme $\cS^*$, which is
\[
(b^*_{l-2j})_{j=-p}^{p+1} = \vec\Lambda_{l}(0,\ldots,0), \qquad l=0,1.
\]
We are going to prove that $\cS[f^0]$ is asymptotically equivalent to $\cS^*$. Note that
\[
\|b^{2i,k} - b^*\|_\infty = \|b^{2i+1,k} - b^*\|_\infty = \max_{l=0,1} \{|\vec\Lambda_{l}(\pi_k(f^k_{i-p},\ldots,f^k_{i+p+1})) - \vec\Lambda_{l}(0,\ldots,0)|_\infty\}.
\]
By hypothesis,
\[
\|b^{2i,k} - b^*\|_\infty = \|b^{2i+1,k} - b^*\|_\infty \leq C |\pi_k(f^k_{i-p},\ldots,f^k_{i+p+1})|_\infty \leq 2^{-k} C\rho.
\]
Finally, this implies that
\[
\sum_{k=0}^\infty \sup_{i\in\Z} \{\|b^{i,k} - b^*\|\} \leq \sum_{k=0}^\infty 2^{-k} C\rho < \infty.
\]
\end{proof}
The constants $\rho,C$ could be \emph{absorbed} by the vector norms and, thus, removed from the definition of class B. We keep them for the sake of clarity. 

\begin{cor}
	Any scheme belonging to class B is convergent.
\end{cor}
\begin{proof}
	Apply Theorem \ref{teo:nonlinear_asymtotic_equivalent} and Corollary \ref{cor:quasilinear_propertyA_implies_Cm}.
\end{proof}

\begin{prop} \label{prop:ours_classB}
	For any $\rho\geq0$, the subdivision scheme in Definition \ref{defi:scheme} belongs to class B.
\end{prop}
\begin{proof}
	Taking $q=2$, $p=1$,
	\begin{align}
		\pi_k(f_{-1},f_0,f_1,f_2) &= ((\alpha^k,\beta^k)\circ(A,B))(\nabla f_{-1},\nabla f_{0},\nabla f_{1}) - (1,1)  \label{eq:pi_k_ours}\\
		\begin{split} \label{eq:vecLambda}
		\vec\Lambda_{0}(\pi) &= (a_{-1},a_{0},a_1,a_2)(\pi_1+1,\pi_2+1),\\
		\vec\Lambda_{1}(\pi) &= (a_{2},a_1,a_0,a_{-1})(\pi_2+1,\pi_1+1).
		\end{split}
	\end{align}
	Observe that we are using function evaluations to shorten the notation in last equations.
	
	The values that $\alpha^k,\beta^k$ can reach are contained in $[(1+2^{-k}\rho)^{-1},1+2^{-k}\rho]$ by definition and, as a consequence,
	\[
	|\alpha^k(A,B) - 1|,|\beta^k(A,B) - 1| \leq 2^{-k}\rho, \qquad \forall A,B>0.
	\]
	Hence, $|\pi_k(f_{-1},f_0,f_1,f_2)|_\infty \leq 2^{-k}\rho$, $D = [(1+2^{-k}\rho)^{-1}-1,2^{-k}\rho]^2$ and $(0,0)\in D$.
	
	Observe that $a_i$ are Lipschitz operators in $[(1+2^{-k}\rho)^{-1},1+2^{-k}\rho]^2$, then $\vec\Lambda_{l}$, $l=0,1$, is Lipschitz in $D$. Thus, taking as $C$ the Lipschitz constant we have that
	\[
	|\vec\Lambda_{l}(\pi) - \vec\Lambda_{l}(0,0)|_\infty \leq  C|\pi|_\infty, \quad \pi\in D, \ l=0,1.
	\]
	
	Finally, as commented in Remark \ref{rmk:dyn_scheme}, we have that $\cS^*$ is convergent.
\end{proof}

Note that if $\cS$ belongs to class B, it may not be scalar-based, as is the case with the proposed scheme. However, as demonstrated in Theorem \ref{teo:nonlinear_asymtotic_equivalent}, $\cS[f^0]$ is scalar-based for any $f^0\in\liZn$.

\begin{teo} \label{teo:classB_C1}
Let $\cS$ be belonging to class B
such that
\(
\langle\vec\Lambda_{l}(\pi),(1,\ldots,1)\rangle = 1, \, \forall \pi\in D, \ l=0,1.
\)
Then $\cS$ reproduces constants. In addition, if $\cS^*$ is $\cC^1$, then $\cS$ is $\cC^1$.
\end{teo}
\begin{proof}
	Let $f^0 = (\ldots,c,c,c,\ldots)$, $c\in\R^n$, be some constant sequence. We will prove by induction on $k\in\N_0$ that $f^k = f^0$.
	Using the definition of class B and the hypothesis, we have that
	\begin{align*}
	f^{k+1}_{2i+l} &= \Psi^k_l(f^k_{i-p},\ldots,f^k_{i+p+1}) = \langle\vec\Lambda_{l}(\pi_k(c,\ldots,c)),(c,\ldots,c)\rangle \\
	&= c\ \langle\vec\Lambda_{l}(\pi_k(c,\ldots,c)),(1,\ldots,1)\rangle = c, & l=0,1, \ i\in\Z.
\end{align*}

To prove that $\cS$ is $\cC^1$, by Corollary \ref{cor:quasilinear_propertyA_implies_Cm}, we have prove that $\cS[f^0]$ satisfies property A of order 1, $\forall f^0\in\liZn$. Let us consider that $f^0$ is fixed and $f^k$ is computed by $\cS$. We denote by $\{b^{i,k}\}_{i\in\Z}$ the masks of $S_k[f^k]$. We have to prove that
\[
\sum_{k=0}^\infty 2^{k}\left | d^{i,k}_1(\pm 1) \right | < \infty, \quad i\in\Z.
\]
By \eqref{eq:mask_classB}, the odd and even positions of the masks of $b^{i,k}$ sums 1, since
\[
\sum_{j\in\Z}b^{2i,k}_{l-2j} = \sum_{j=-p}^{p+1} b^{2i,k}_{l-2j} = \langle (b^{2i,k}_{l-2j})_{j=-p}^{p+1},(1,\ldots,1)\rangle = \langle\vec\Lambda_{l}(\pi_k(f^k_{i-p},\ldots,f^k_{i+p+1})),(1,\ldots,1)\rangle = 1, \qquad l=0,1,
\]
and $b^{2i,k} = b^{2i+1,k}$.
By Lemma \ref{lem:sum_one_implies_d0}, we have that $d^{i,k}_1(\pm 1) = 0$, so that the property A is clearly satisfied.
\end{proof}

\begin{cor} \label{cor:scheme_C1}
	For any $\rho\geq0$, the subdivision scheme in Definition \ref{defi:scheme} is $\cC^1$.
\end{cor}
\begin{proof}
We check that fulfils Theorem \ref{teo:classB_C1}. It is of class B by Proposition \ref{prop:ours_classB} and the coordinates of $\vec\Lambda_{l}(\pi)$ sum 1 by \eqref{eq:vecLambda} and by Remark \ref{rmk:lagrange_affine}.
\end{proof}

\section{Convergence by non-linear subdivision theory} \label{sec:convergence_nonlinear}

The truncation procedure in \eqref{eq:alpha_beta_k} have been intentionally designed to make the scheme non-stationary and `asymptotically equivalent' to a linear scheme. This property simplifies the process of demonstrating convergence, in contrast to the potential complexity involved in proving the convergence of the stationary version of the scheme (consisting of removing the truncation procedure). The latter has a large non-linear expression that might be challenging to manage using conventional analytical tools for non-linear schemes.

In addition, the subdivision operators of the proposed scheme belong to the class of non-linear operators studied in \cite{ADL11}. Subsequently, we adapt \cite[Theorem 1]{ADL11} to our non-stationary context.

\begin{teo}\label{teo:convergence}
	Let $\cS= \{S_k:\liZn\lra \liZn\}_{k\in\N_0}$ be a subdivision scheme such that
	\begin{equation} \label{eq:class_perturbation}
	S_k = T + \cF_k \circ \nabla, \qquad k\in\N_{0},
	\end{equation}
	where $\cF_k: \liZn\lra \liZn$ is a non-linear operator and the stationary subdivision scheme defined by the subdivision operator $T$ is linear and convergent.
	If there exists $M>0$, $c\in(0,1)$, $k_0,K\in\N_0$ such that, $\forall k\geq k_0$,
	\begin{align}
		\| \cF_k \delta \|_\infty &\leq M \|\delta\|_\infty, & \forall \delta \in \liZn, \label{eq:cF_bounded}\\
		\|\nabla S_{k+K} \cdots S_k f \|_\infty &\leq c \|\nabla f\|_\infty, & \forall f \in \liZn, \label{eq:contractivity}
	\end{align}
	then the subdivision scheme converges.
\end{teo}
\begin{proof}
	This result was proven in \cite{ADL11} for scalar-valued ($n=1$) and stationary ($S_k=S_0$, $\forall k\in\N_0$, $k_0=0$) subdivision schemes.
	However, the proof would be fairly the same for $n>1$ and the non-stationary case. In fact, the theorem in \cite{ADL11} rely on the original results in \cite{DRS04}, where the non-stationary case and $n=2$ was considered.
	
	Furthermore, the initial iterations have no impact on the convergence of a subdivision scheme, as we can regard $f^{k_0}$ as the initial data. For the subsequent iterations, the hypotheses remain valid.
\end{proof}

\begin{teo} \label{teo:convergence_perturbation}
	Let $\cS$ be belonging to class B such that
	\(
	\langle\vec\Lambda_{l}(\pi),(1,\ldots,1)\rangle = 1, \, \forall \pi\in D, \ l=0,1.
	\)
	Suppose that there exists $\pi^{[1]}_k:(\R^n)^{2p+1} \lra D$ such that \(\pi_k(f_{-p},\ldots,f_{p+1}) = \pi^{[1]}_k(\nabla f_{-p},\ldots,\nabla f_{p})\).
	Then, $\cS$ fulfils the hypothesis of Theorem \ref{teo:convergence} (thus, it converges).
\end{teo}
\begin{proof}
	First, we prove that $\cS$ can be written as in \eqref{eq:class_perturbation}.
	By hypothesis,
	\(
	\langle\vec\Lambda_{l}(\pi)-\vec\Lambda_{l}(0),(1,\ldots,1)\rangle = 0, \, \forall \pi\in D, \ l=0,1.
	\)
	Hence,
	\[\vec\Lambda_{l}(\pi)-\vec\Lambda_{l}(0) \in \espan{(1,1,\ldots,1,1)}^\perp
	= \espan{(-1,1,0,0,\ldots),(0,-1,1,0,\ldots),\ldots,(\ldots,0,-1,1)}.
	\]
	Then, there exist some coefficients $\vec\Lambda^{[1]}_{l}(\pi) = ((\vec\Lambda^{[1]}_{l}(\pi))_j)_{j=-p}^p \in\R^{2p+1}$ such that
	\[
	\vec\Lambda_{l}(\pi)-\vec\Lambda_{l}(0) = \sum_{j=-p}^p (\vec\Lambda^{[1]}_{l}(\pi))_j (0,\ldots,\overset{(j)}{-1},1,0,\ldots,0), \qquad l=0,1.
	\]
	This means that
	\begin{align*}
		\langle\vec\Lambda_{l}(\pi)-\vec\Lambda_{l}(0),(f_{-p},\ldots,f_{p+1})\rangle
	&=\sum_{j=-p}^p (\vec\Lambda^{[1]}_{l}(\pi))_j \langle (0,\ldots,\overset{(j)}{-1},1,0,\ldots,0),(f_{-p},\ldots,f_{p+1})\rangle \\
	&=
	\sum_{j=-p}^p (\vec\Lambda^{[1]}_{l}(\pi))_j \nabla f_{j}
	=
	\langle\vec\Lambda^{[1]}_{l}(\pi),(\nabla f_{-p},\ldots,\nabla f_{p})\rangle.
	\end{align*}
	Now, for $l=0,1$, by Definition \ref{defi:classB} (class B), by the latter equation and by hypothesis, in this order:
	\begin{align*}
		\Psi^k_l(f_{-p},\ldots,f_{p+1}) &=  \langle \vec\Lambda_{l}(\pi_k(f_{-p},\ldots,f_{p+1})),(f_{-p},\ldots,f_{p+1})\rangle\\
		&=  \langle \vec\Lambda_{l}(0),(f_{-p},\ldots,f_{p+1})\rangle
		+  \langle \vec\Lambda_{l}(\pi_k(f_{-p},\ldots,f_{p+1}))-\vec\Lambda_{l}(0),(f_{-p},\ldots,f_{p+1})\rangle\\
		&=  \langle \vec\Lambda_{l}(0),(f_{-p},\ldots,f_{p+1})\rangle
		+  \langle\vec\Lambda^{[1]}_{l}(\pi_k(f_{-p},\ldots,f_{p+1})),(\nabla f_{-p},\ldots,\nabla f_{p})\rangle\\
		&=  \langle \vec\Lambda_{l}(0),(f_{-p},\ldots,f_{p+1})\rangle
		+  \langle\vec\Lambda^{[1]}_{l}(\pi^{[1]}_k(\nabla f_{-p},\ldots,\nabla f_{p})),(\nabla f_{-p},\ldots,\nabla f_{p})\rangle.
	\end{align*}
	Thus, the subdivision operator can be written as $S_k = T + \cF_k \circ \nabla$, where $T$ coincides with the subdivision operator of $\cS^*$ (its coefficients are given by $\vec\Lambda_{l}(0)$, $l=0,1$) and $\cF_k: \liZn\lra \liZn$ is defined as
	\begin{align} \label{eq:def_Fk}
		(\cF_k(\delta))_{2i + l} &:= \langle\vec\Lambda^{[1]}_{l}(\pi^{[1]}_k(\delta_{i-p},\ldots,\delta_{i+p})),(\delta_{i-p},\ldots,\delta_{i+p})\rangle, & l=0,1.
	\end{align}

	Second, we prove that \eqref{eq:cF_bounded} is satisfied.
	By definition of class B, taking the vector norm $|\cdot|_{\R^{q}} := |\cdot|_\infty$ and considering
	\begin{equation} \label{eq:choice_of_f}
		(f_{-p},\ldots,f_{p+1}) := (0,\delta_{-p},\delta_{-p}+\delta_{-p+1},\ldots,\sum_{j=-p}^p\delta_{j}),
	\end{equation}
	so that $(\nabla f_{-p},\ldots,\nabla f_{p}) = (\delta_{-p},\ldots,\delta_{p})$, we have that
	\begin{align} \label{eq:pi1_bounded}
	| \pi^{[1]}_k(\delta_{-p},\ldots,\delta_{p}) |_\infty = | \pi^{[1]}_k(\nabla f_{-p},\ldots,\nabla f_{p}) |_\infty = |\pi_k(f_{-p},\ldots,f_{p+1})|_\infty \leq 2^{-k}\rho.
	\end{align}
	That is, $\pi^{[1]}_k$ is bounded and converges uniformly to 0 when $k\to\infty$. We can also deduce that
	\begin{align} \label{eq:teo:convergence:1}
		\langle\vec\Lambda^{[1]}_{l}(\pi),(\delta_{-p},\ldots,\delta_{p})\rangle = 
		\langle\vec\Lambda^{[1]}_{l}(\pi),(\nabla f_{-p},\ldots,\nabla f_{p})\rangle = \langle\vec\Lambda_{l}(\pi)-\vec\Lambda_{l}(0),(f_{-p},\ldots,f_{p+1})\rangle.
	\end{align}
	Observe that the last expression is an equality between vectors in $\R^{n}$. Now we only pay attention to one of the coordinates, the first one, for example.
	For each $j\in\{-p,\ldots,p\}$, we choose the data such that $(\delta_j)_1 = 1$ and $(\delta_i)_1 = 0$, $i\neq j$. By \eqref{eq:choice_of_f},
	\[
		((f_{-p})_1,\ldots,(f_{p+1})_1) = (0,0,\ldots,0,\overset{(j)}{1},1,\ldots,1)
	\]
	and, by \eqref{eq:teo:convergence:1},
	\begin{align*}
	|(\vec\Lambda^{[1]}_{l}(\pi))_j|
	&= |(\langle\vec\Lambda^{[1]}_{l}(\pi),(\delta_{-p},\ldots,\delta_{p})\rangle)_1|
	= |(\langle\vec\Lambda_{l}(\pi)-\vec\Lambda_{l}(0),f_{-p},\ldots,f_{p+1}\rangle)_1| \\
	&= |\langle\vec\Lambda_{l}(\pi)-\vec\Lambda_{l}(0),((f_{-p})_1,\ldots,(f_{p+1})_1)\rangle|
	= |\langle\vec\Lambda_{l}(\pi)-\vec\Lambda_{l}(0),(0,0,\ldots,1,1\ldots)\rangle| 
	\\
	& \leq |\vec\Lambda_{l}(\pi)-\vec\Lambda_{l}(0)|_1 \leq C|\pi|_\infty.
	\end{align*}
	Observe that we have conveniently considered in the last line that $|\cdot|_{\R^{2p+2}}:=|\cdot|_1$.
	Then $|\vec\Lambda^{[1]}_{l}(\pi)|_\infty \leq C|\pi|_\infty$, $l=0,1$. With this, \eqref{eq:pi1_bounded} and \eqref{eq:def_Fk}, we conclude that
	\begin{align*}
		|(\cF_k(\delta))_{2i + l}| &\leq |\vec\Lambda^{[1]}_{l}(\pi^{[1]}_k(\delta_{i-p},\ldots,\delta_{i+p}))|_1 \|(\delta_{i-p},\ldots,\delta_{i+p})\|_\infty \\
		&\leq (2p+1)|\vec\Lambda^{[1]}_{l}(\pi^{[1]}_k(\delta_{i-p},\ldots,\delta_{i+p}))|_\infty \|(\delta_{i-p},\ldots,\delta_{i+p})\|_\infty 
		\\
		&\leq (2p+1)C |\pi^{[1]}_k(\delta_{i-p},\ldots,\delta_{i+p})|_\infty \|(\delta_{i-p},\ldots,\delta_{i+p})\|_\infty 
		\leq (2p+1)C 2^{-k} \rho\|(\delta_{i-p},\ldots,\delta_{i+p})\|_\infty.
	\end{align*}
That is, $	\| \cF_k \delta \|_\infty \leq M_k \|\delta\|_\infty$ with $M_k=(2p+1)C 2^{-k} \rho \leq (2p+1)C \rho$.
	In addition, we conclude that the sequences of operators $\{\cF_k\}_{k\in\N_0}$ converges uniformly to zero.
	
	Finally, we prove \eqref{eq:contractivity}. Observe that,
	\[
	\nabla \circ S_k = \nabla \circ T + \nabla\circ\cF_k\circ\nabla = T^{[1]}\circ  \nabla + \nabla\circ\cF_k\circ\nabla = (T^{[1]} + \nabla\circ\cF)\circ\nabla = S_{k}^{[1]}\circ\nabla,
	\]
	where $T^{[1]}$ is referred to as the \emph{difference operator} of $T$, and $S_{k}^{[1]} := T^{[1]} + \nabla\circ\cF_k$ turns out to be the difference operator of $S_k$.
	
	It is well known (see \cite{Dyn92}) that for a linear, stationary, uniform and convergent scheme, as the one defined by $T$, there exists $K\in\N$ such that
	\[
	\| T^{[1]}\circ \overset{(K)}{\cdots} \circ T^{[1]} \|_\infty = c < 1.
	\]
	Using the uniform convergence of $\cF_k$ to zero, we deduce that
\begin{align*}
	\lim_{k\to+\infty} \|\nabla S_{k+K} \cdots S_k f \|_\infty &= \lim_{k\to+\infty} \|S^{[1]}_{k+K} \cdots S^{[1]}_k \nabla f \|_\infty = \lim_{k\to+\infty} \|
	(T^{[1]} + \nabla\circ\cF_{k+K}) \cdots (T^{[1]} + \nabla\circ\cF_{k}) \nabla f \|_\infty \\
	&= \|
	T^{[1]}\circ \overset{(K)}{\cdots} \circ T^{[1]} \nabla f \|_\infty \leq c \|\nabla f \|_\infty.
\end{align*}
Hence, for $c' = (1+c)/2 < 1$, there exists $k_0\in\N$ such that
\[
\|\nabla S_{k+K} \cdots S_k f \|_\infty \leq c' \|\nabla f\|_\infty, \qquad \forall k\geq k_0.
\]
\end{proof}

\section{Numerical experiments} \label{sec:numerical}

In this section, we introduce several numerical experiments. They have been designed to illustrate the reproduction of parabolas on non-uniform grids and the generation of 2D and 3D curves using the proposed scheme. We consider various values of $\rho$ to demonstrate the impact of this parameter on the final shape. For comparison, we also consider $\rho = 0$. In this case, the proposed scheme coincides with the Lagrange-type linear uniform scheme of Remark \ref{rmk:dyn_scheme}.
All computations were executed using MATLAB R2022a. The codes necessary to replicate these experiments are readily available (see section Reproducibility).

\subsection{Parabolas reproduction}

\begin{figure}[h]
	\centering
	\includegraphics[clip,height = 0.23\textheight]{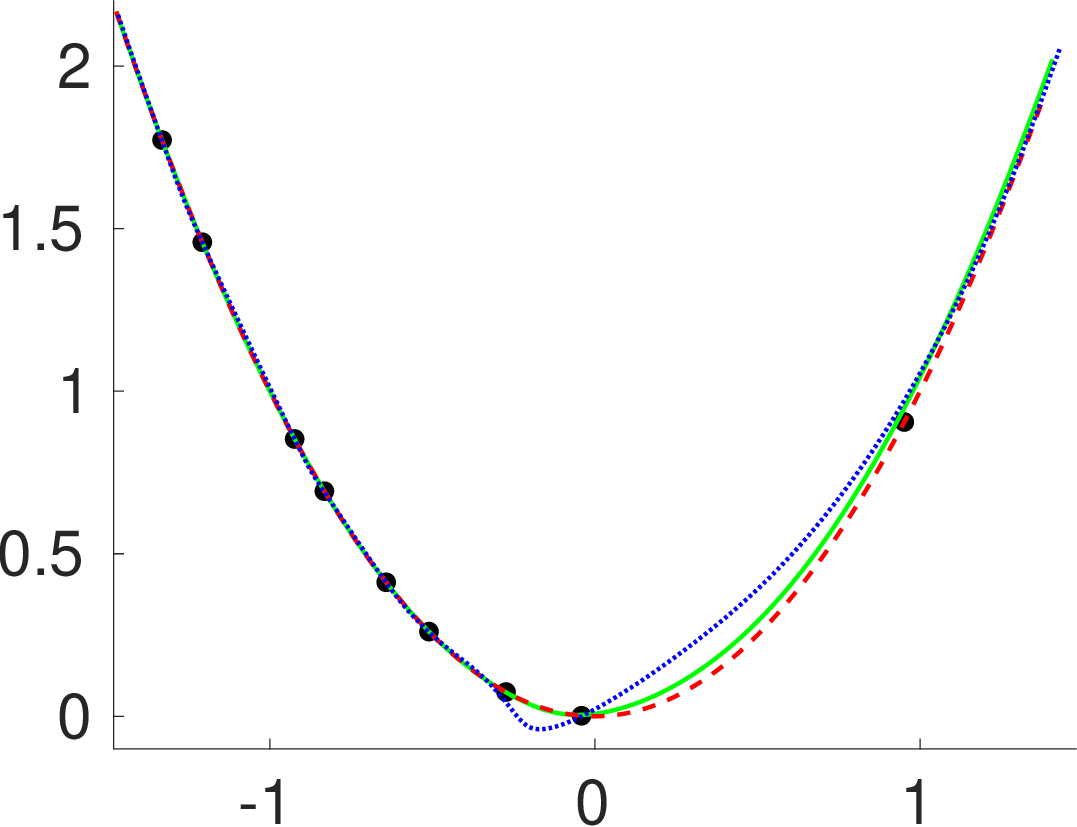} \hspace{10pt}
	\includegraphics[clip,height = 0.23\textheight]{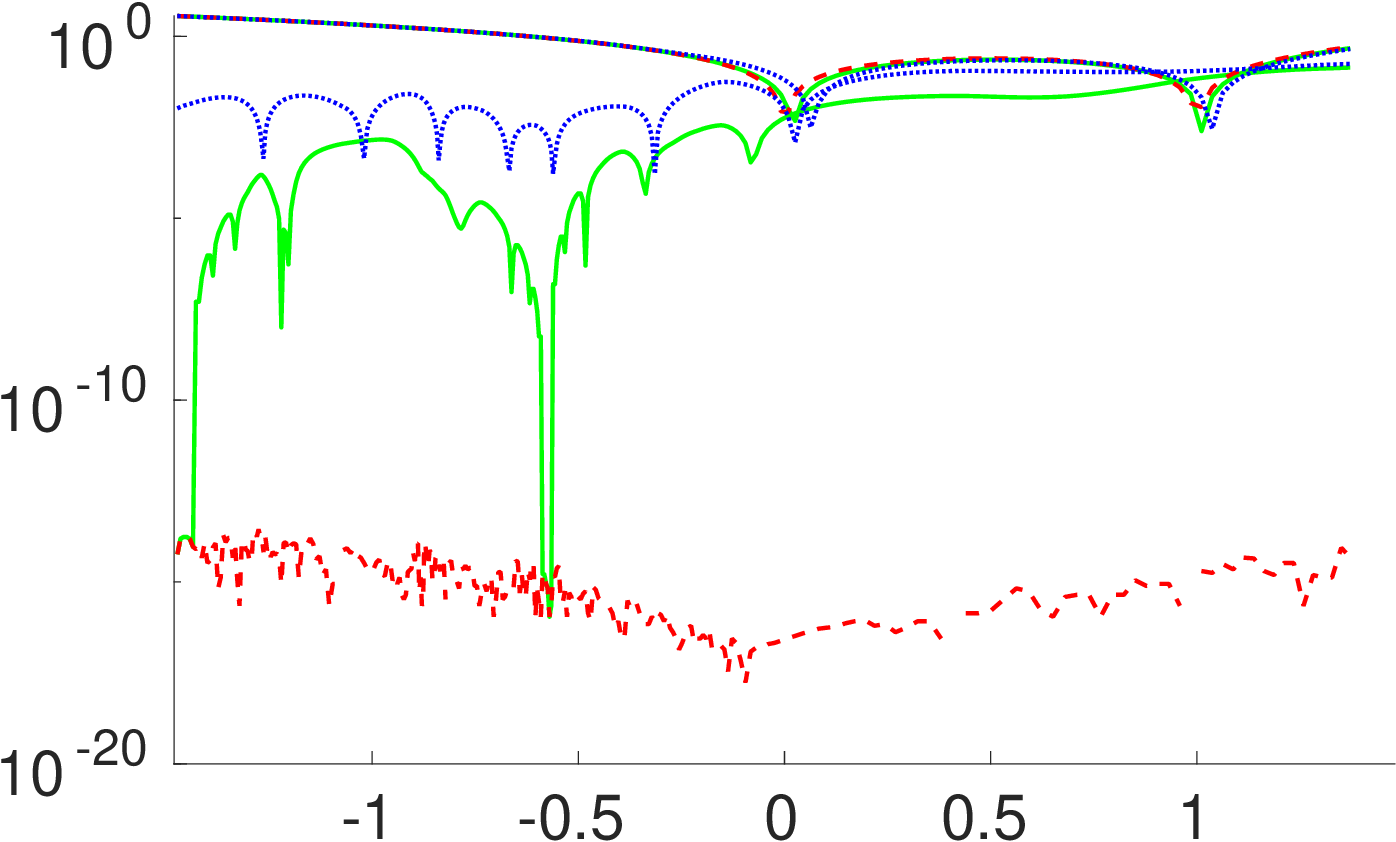}
	\caption[Parabola reproduction]{Left, application of 5 iterations of the scheme with $\rho\in\{0,2,6\}$ (dotted blue, solid green and dashed red, respectively) to parabolic non-uniform data (black dots). Right, the reproduction error in log-scale.}
	\label{fig:parabola}
\end{figure}

We start by confirming the scheme ability to reproduce second degree polynomials. To show this, we choose a function with a commonly recognized graph: $F(x) = (x, x^2)$. It is worth noting that this experiment is applicable to any $F\in\Pi_2^n$, where $n\geq 2$, yielding comparable outcomes.

We randomly generated a non-uniform grid $\xi^0 = (\xi^0_i)_{i=1}^{15}$ within the interval $[-2, 2]$,
\begin{align*}
\xi^{0}= (&-1.9381,   -1.7893,   -1.5751,   -1.3313,   -1.2075,   -0.9235,   -0.8321,\\
  &-0.6420,   -0.5104,   -0.2734,   -0.0412,    0.9514,    1.6813,    1.8065,    1.9363),
\end{align*}
for which
$$\alpha_{11}^{0} = 0.2339\leq\alpha_{i}^{0} = \frac{\nabla \xi_{i-1}}{\nabla \xi_i}\leq \alpha_{13}^{0} = 5.8299, \quad \beta_{12}^{0} = 0.1715\leq\beta_{i}^{0} = \frac{\nabla \xi_{i+1}}{\nabla \xi_i}\leq \beta_{10}^{0} = 4.2748, \, \forall i=2,\ldots,13.$$
According to Theorem \ref{teo:scheme_reproduction}, $\rho \geq 4.8299$ must be chosen to achieve an exact reproduction of the parabola.

In Figure \ref{fig:parabola}-left, we present the results of applying the scheme with $\rho\in\{0,2,6\}$, each for 5 iterations. The boundaries of the interval $[-2, 2]$ are not displayed because the data `shrinks' after each iteration, a typical behaviour in subdivision. The scheme with $\rho = 0$ generates a bump near the vertex of the parabola. The curve obtained with $\rho = 2$ is nearer to the parabola and no bumps appear. Only for $\rho = 6$ is an exact reproduction achieved up to machine precision.

On the right side of the figure, the errors in the reproduction, computed as $F|_{\xi^{5}} - S_{5}(F|_{\xi^{0}})$, are displayed on a logarithmic scale. In some small regions (near $x=1$, for instance), the error is precisely 0, which cannot be visualized on a log-scale, producing some gaps in the graphic.

\subsection{2D curve example}

\begin{figure}[h]
	\centering
	\begin{tabular}{cc}
	\includegraphics[clip,height = 0.25\textheight]{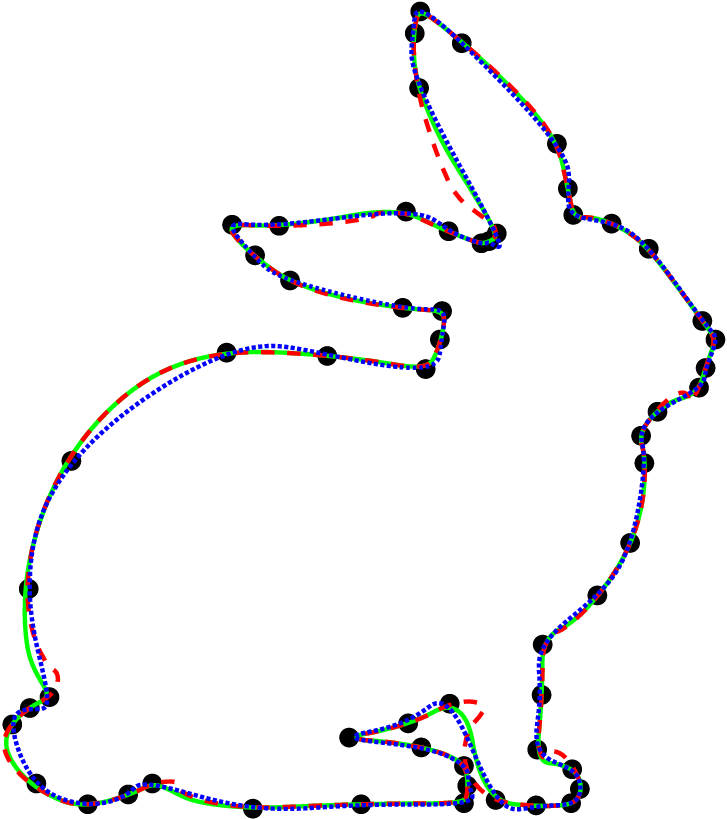}&
	\includegraphics[clip,height = 0.25\textheight]{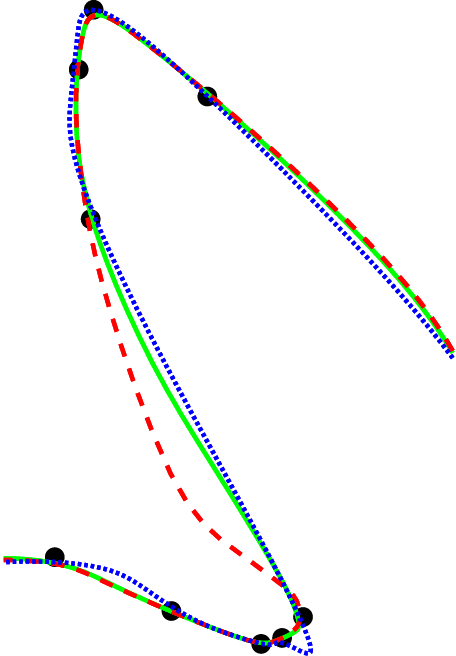}\\
	\includegraphics[clip,height = 0.25\textheight]{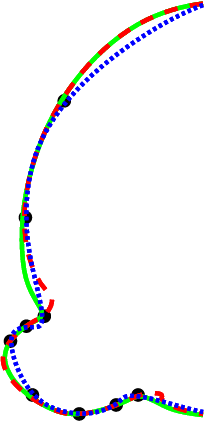}&
	\includegraphics[clip,height = 0.25\textheight]{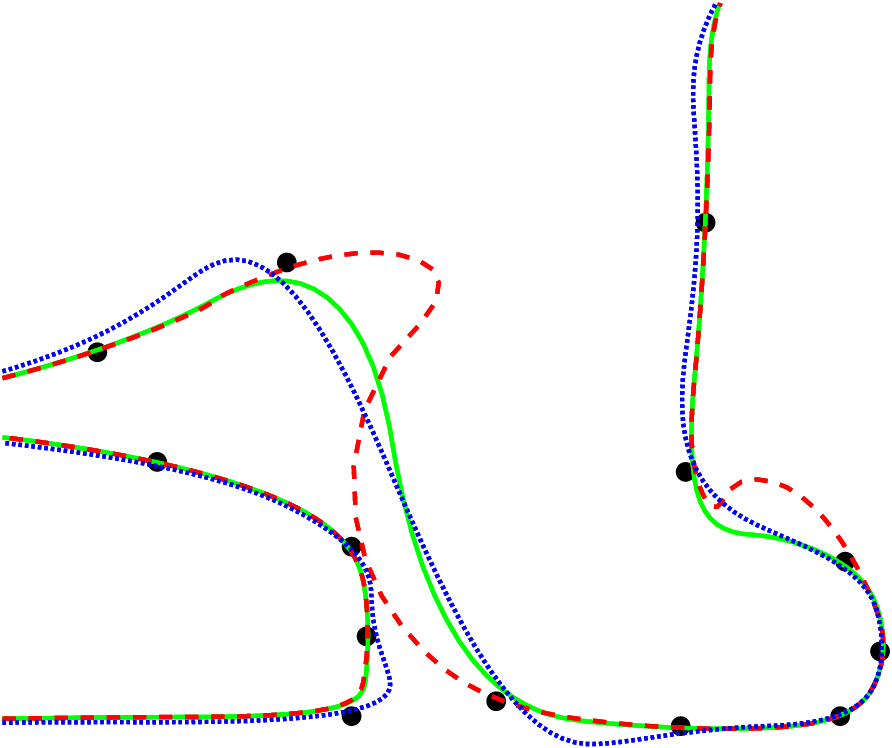}
	\end{tabular}
	\caption[Example of 2D curve]{Example of 2D curves describing a rabbit generated by the presented scheme with $\rho\in\{0,2,6\}$ (dotted blue, solid green and dashed red, respectively) generated from initial data with varying spacing (black dots).
	In the top-left figure, the entire curves can be seen, while the rest of figures are zooms of interesting regions (ear, tail and paw).}
\label{fig:rabbit}
\end{figure}

Next, we present the shapes generated by the scheme when a specific drawing is desired to be created. We have chosen a set of points that outline a rabbit. We applied the scheme for 5 iterations with $\rho\in\{0,2,6\}$. The results can be examined in Figure \ref{fig:rabbit}. Note that the spacing between points varies along the curve, as there are regions that require more detail, while others do not. It motivates the use of a non-uniform-based scheme, as ours, which implicitly examines the distances between the points and suggests the grid spacings for selecting a Lagrange scheme from Definition \ref{defi:lagrange}.

\begin{figure}
	\centering
	\includegraphics[width=0.4\linewidth]{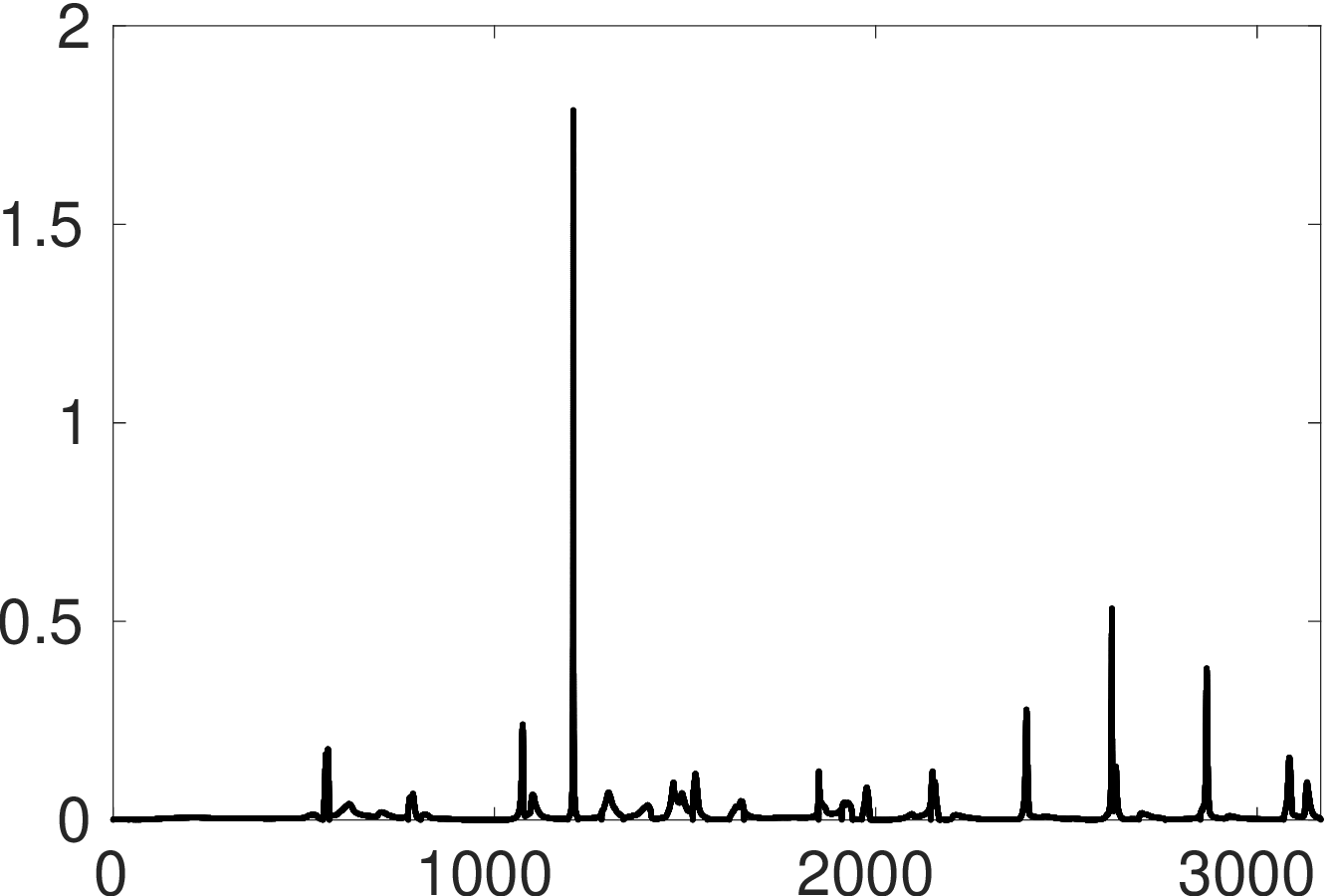}
	\includegraphics[width=0.418\linewidth]{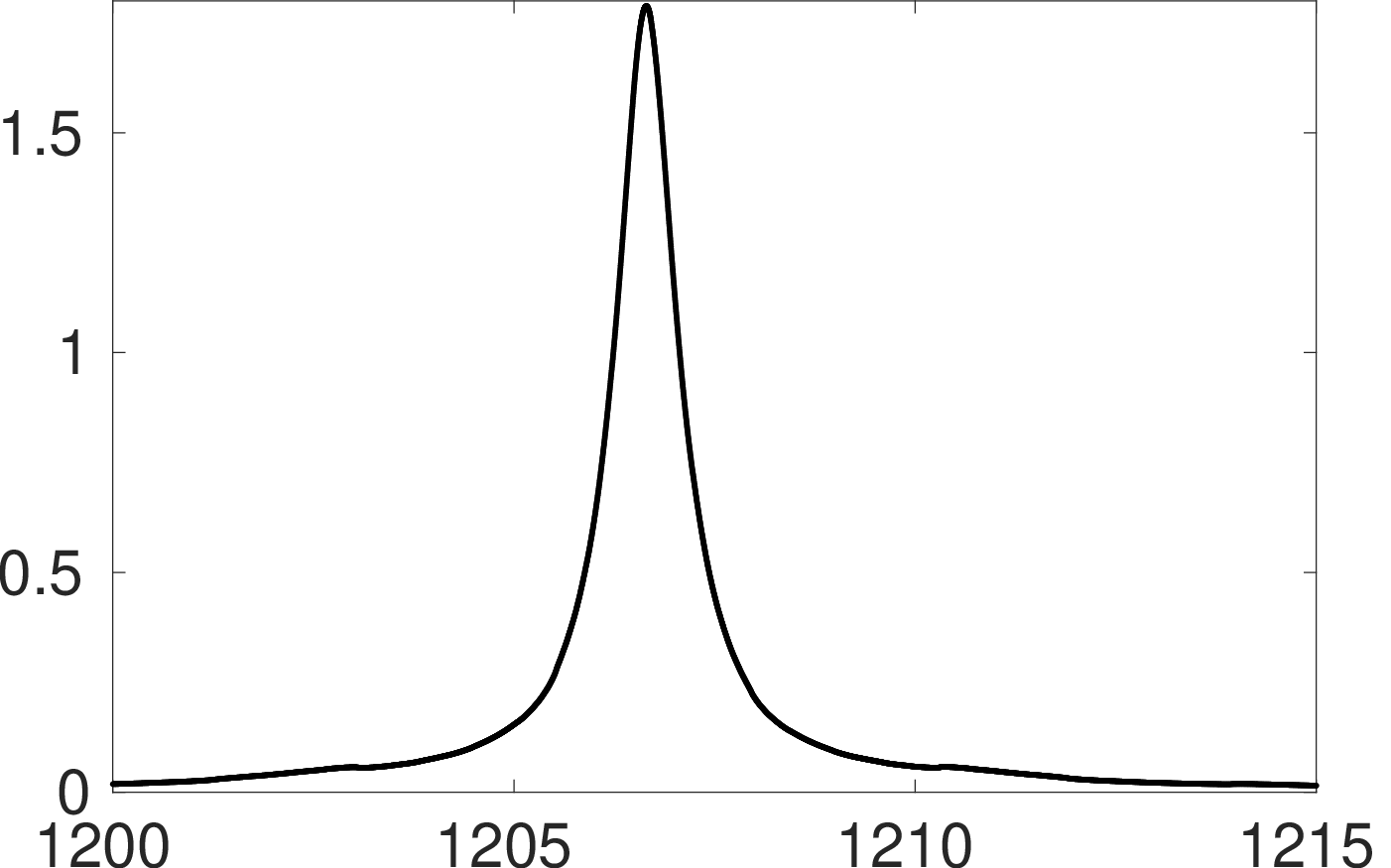}
	\caption{The cumulative arc length versus the curvature of the curve generated by the scheme with $\rho=2$ in Figure \ref{fig:rabbit} is displayed in the left side. A zoom over the highest peak is presented in the right side. The curvature is estimated using the discrete set of points generated by the scheme after 12 iterations, and the curvature estimation at each generated point is marked with a black dot. The numerous dots and the continuity of the curvature along the rabbit curve create the appearance of a continuous function, as evident in the display.}
	\label{fig:rabbitcurvature}
\end{figure}

Note that for $\rho = 0$, the generated curve displays noticeable oscillations in areas such as the ear and paw. Conversely, in other regions like the tail and back, it appears somewhat `stiff'. This highlights the limited adaptability of the linear scheme.
In contrast, the non-linear scheme ($\rho>0$) appears to be more adaptive and flexible. A comparison of the results for $\rho = 2$ and $\rho = 6$ clearly demonstrates that this parameter precisely controls the level of flexibility. It is worth noting the shapes generated for $\rho = 6$ in the ear (which, in our opinion, is also a good representation of an ear), in the tail (which has two undesirable bumps at the beginning and end), and in the paw (where the result is quite undesirable due to an excessive arc length). Good-looking results are obtained for $\rho = 2$.

We computed the curvature along the rabbit curve and determined that, for $\rho\in\{0,2,6\}$, it exhibits continuity. Figure \ref{fig:rabbitcurvature} specifically depicts the case for $\rho=2$. It is worth noting that, despite the presence of peaks where the curvature achieves local maxima, a closer examination reveals that these peaks are also smooth.

\subsection{3D curve example}

\begin{figure}[h]
	\centering
	\includegraphics[width=0.3\linewidth,clip,trim = {70pt 30pt 60pt 20pt}]{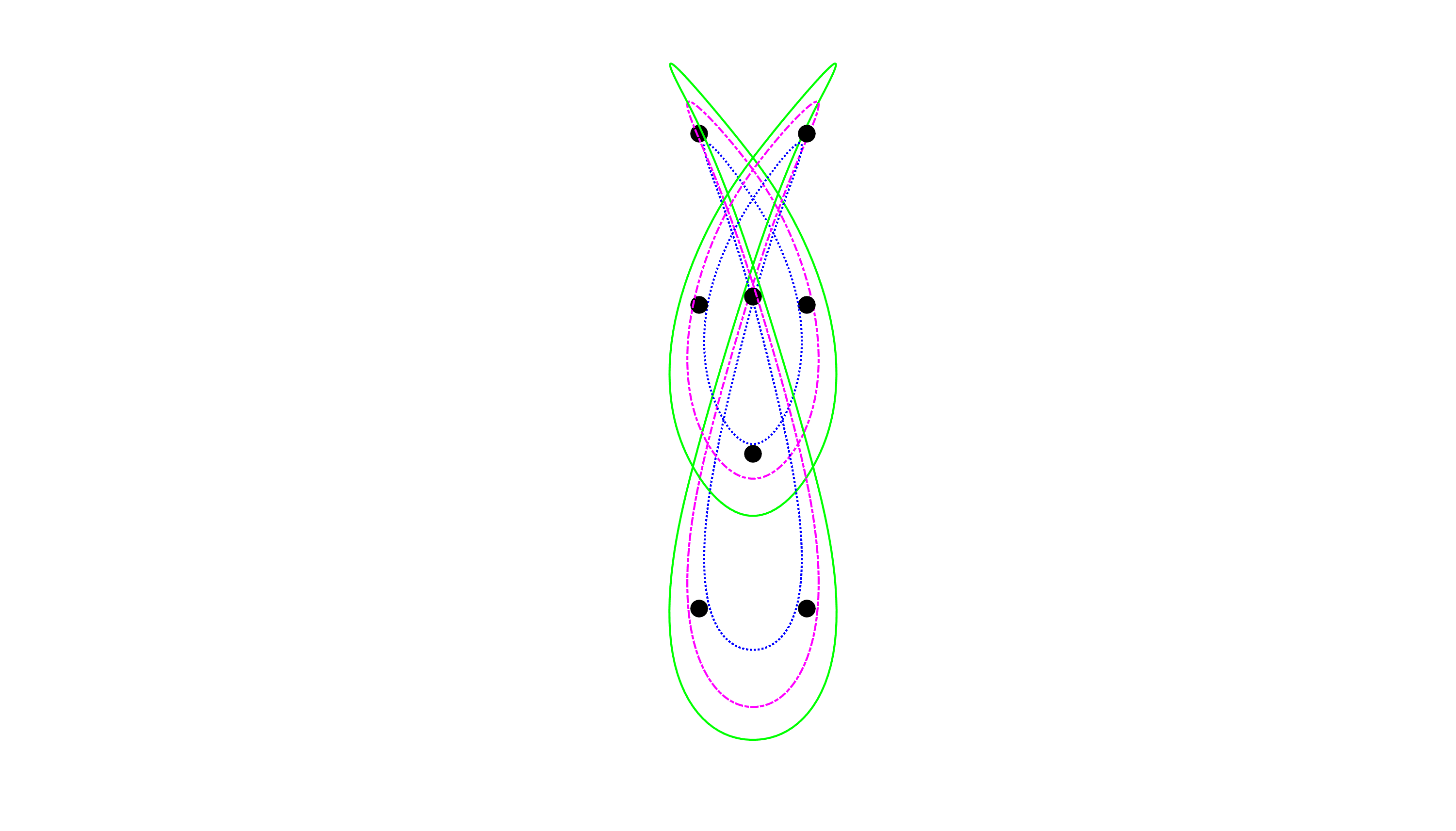}
	\includegraphics[width=0.3\linewidth,clip,trim = {70pt 30pt 60pt 20pt}]{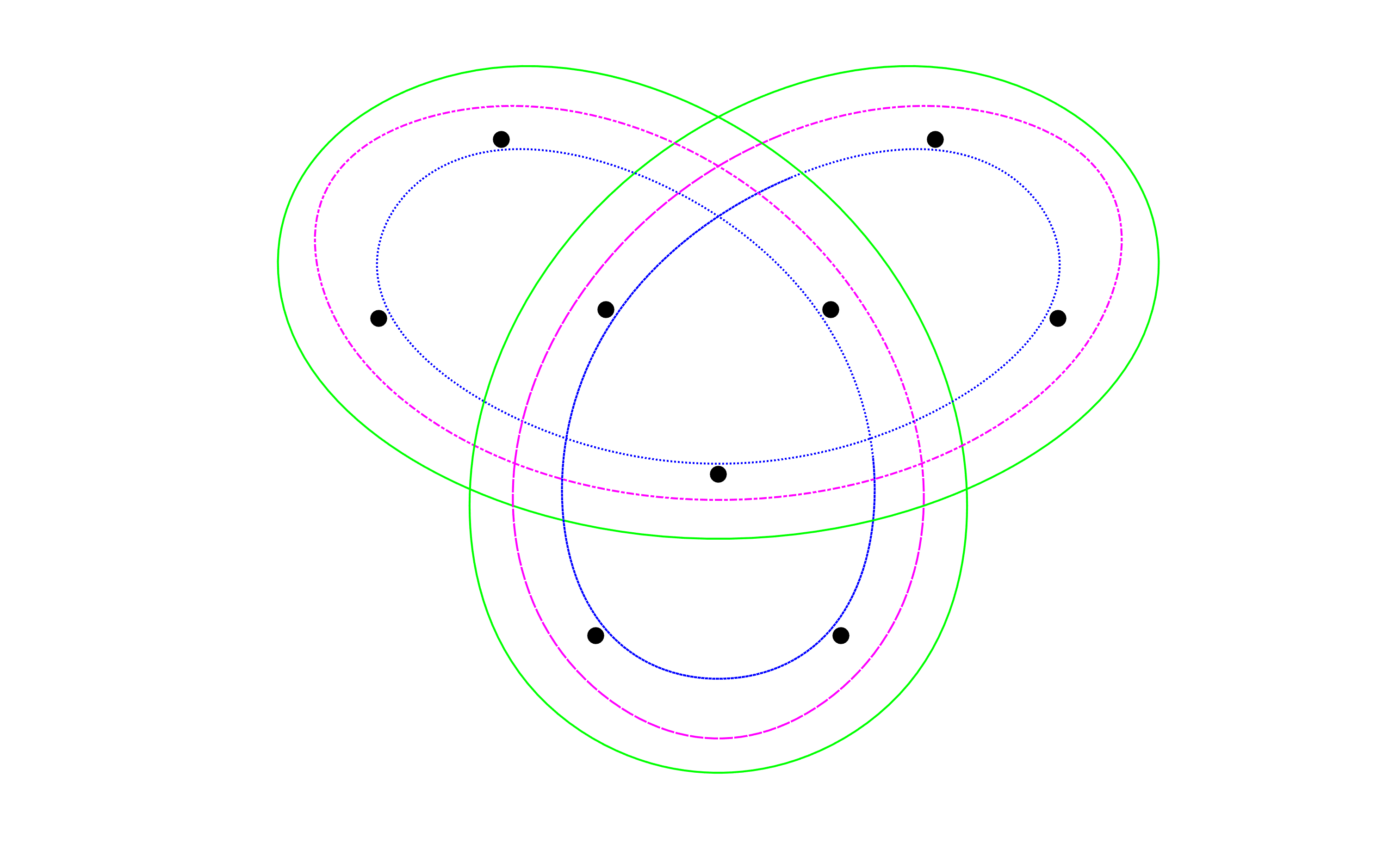}
	\includegraphics[width=0.3\linewidth,clip,trim = {70pt 30pt 60pt 20pt}]{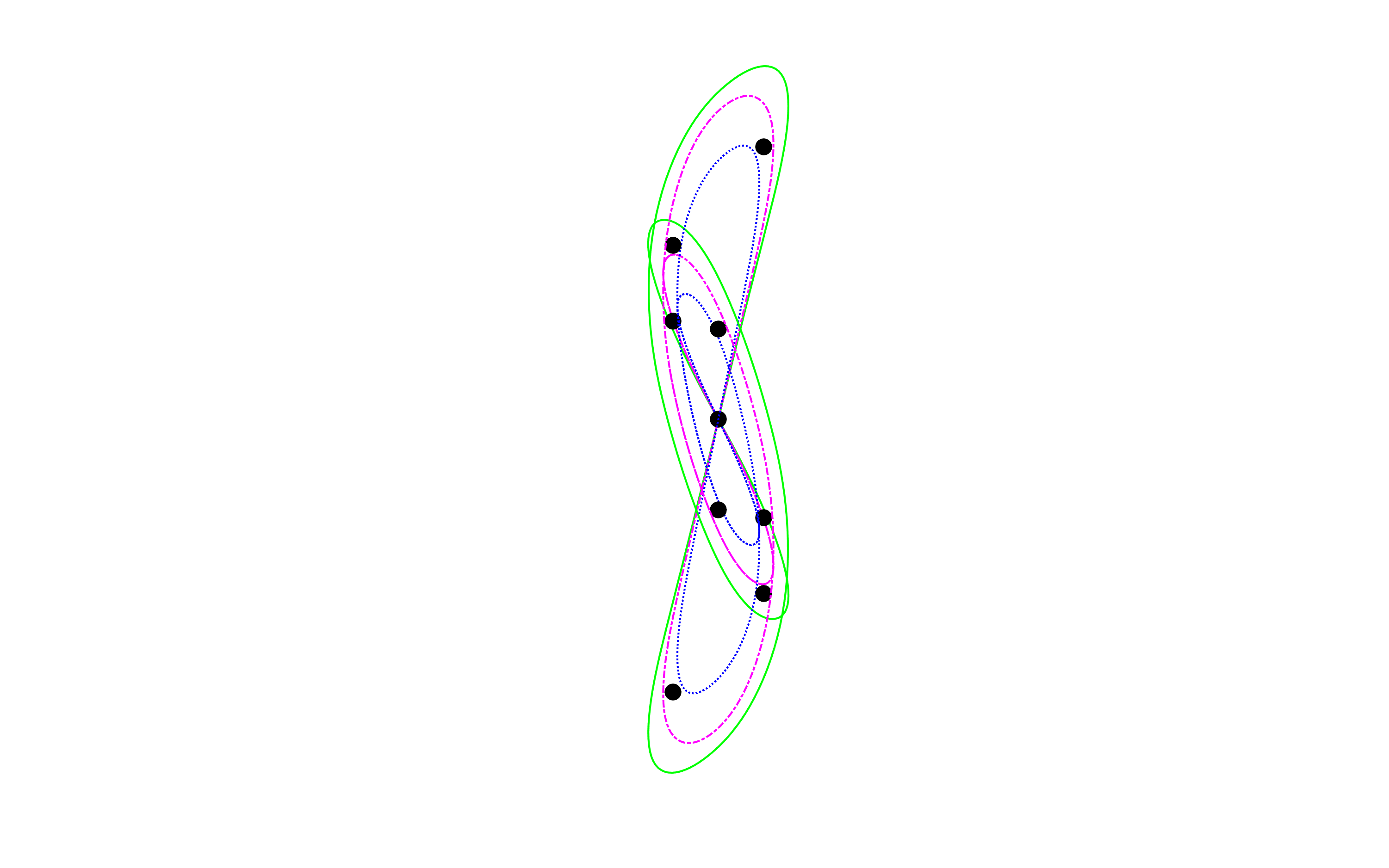}\\[10pt]
	\includegraphics[width=0.3\linewidth,clip,trim = {50pt 30pt 40pt 20pt}]{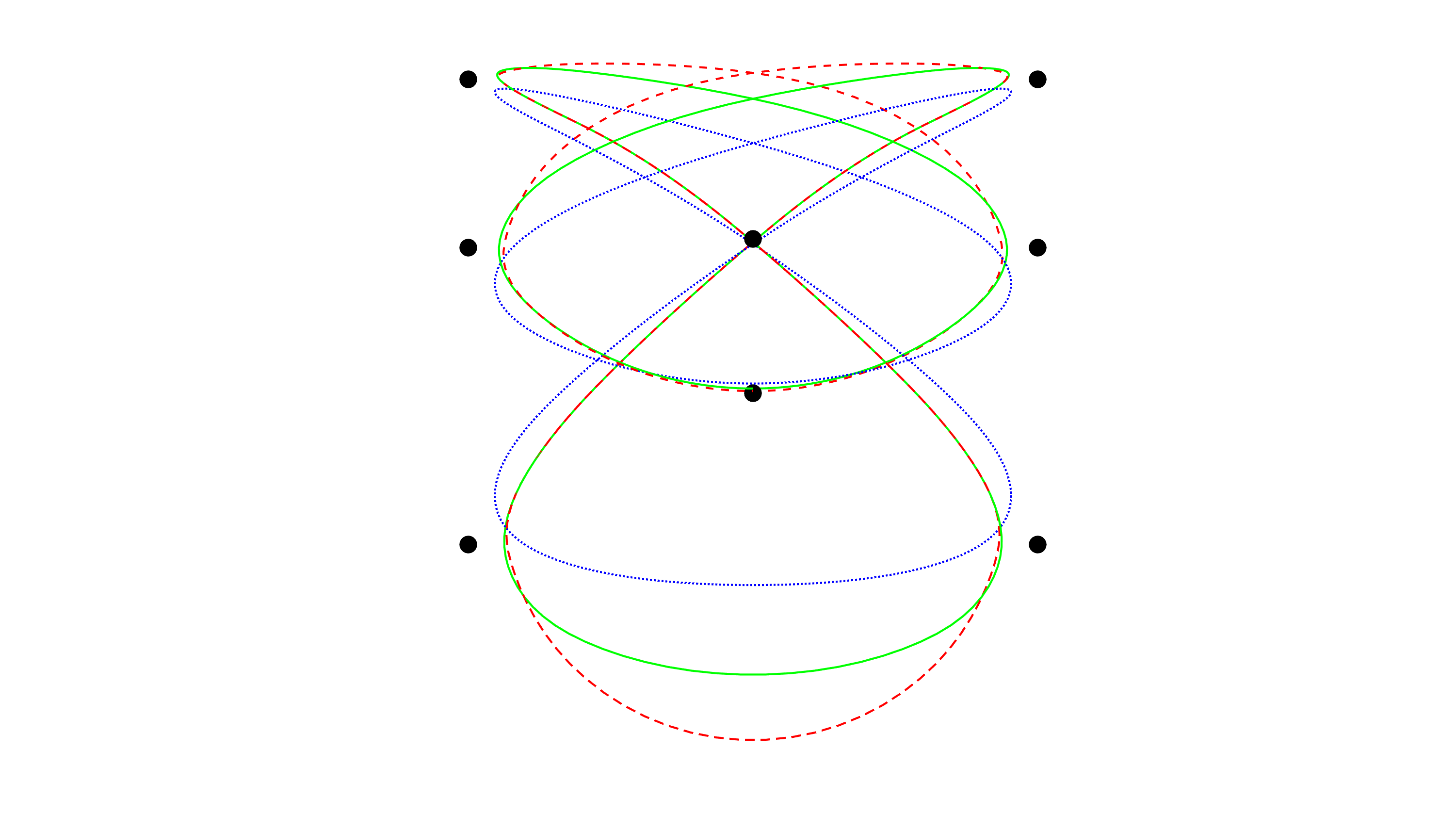}
	\includegraphics[width=0.3\linewidth,clip,trim = {50pt 30pt 40pt 20pt}]{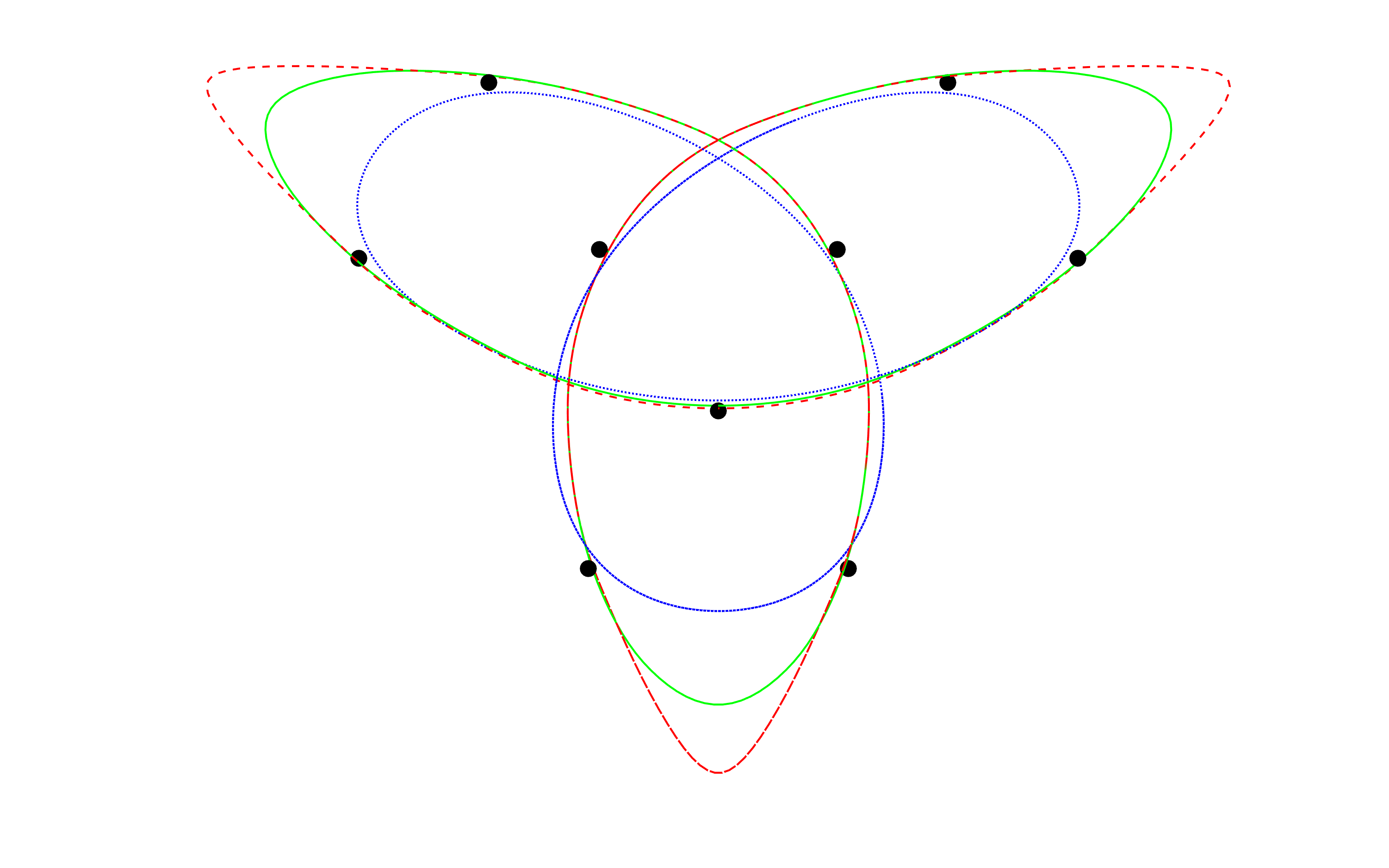}
	\includegraphics[width=0.3\linewidth,clip,trim = {50pt 30pt 40pt 20pt}]{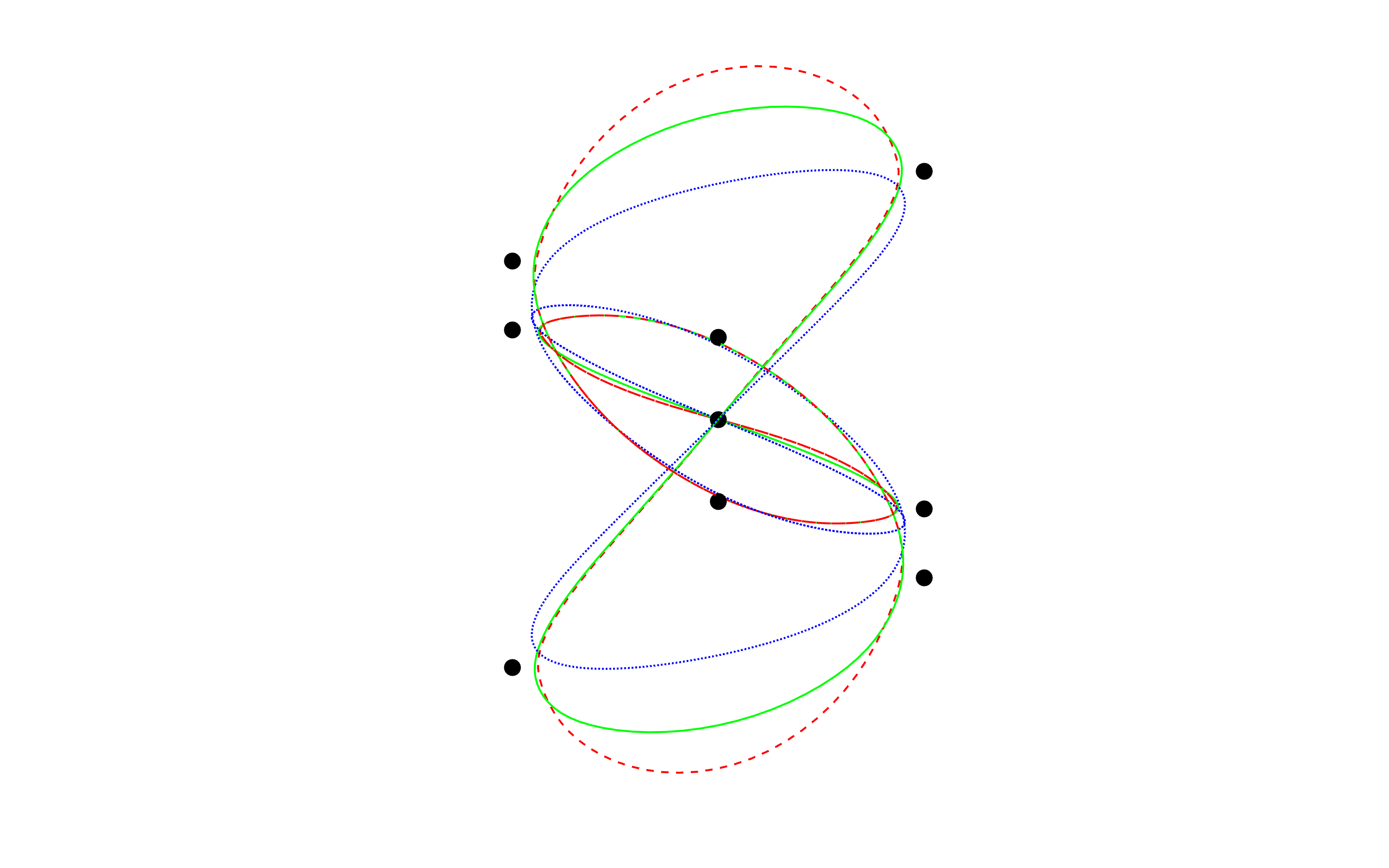}
	\caption[Example of 3D curve]{Example of 3D curves describing trefoil knots, taking $v=0$ (top row) and $v=\pi$ (bottom row) in \eqref{eq:trefoil}, generated by the presented scheme with $\rho\in\{0,1,2\}$ and with $\rho\in\{0,2,6\}$, respectively. The line aspect for each $\rho$ value is: $\rho = 0$, dotted blue; $\rho = 1$, dash-dotted magenta; $\rho = 2$, solid green; and $\rho = 6$, dashed red. We consider the front, lateral and top views of these 3D curves (left, central and right columns, respectively).}
	\label{fig:trefoil}
\end{figure}

Now we proceed with a 3D curve example. Let us consider a trefoil knot surface described by the parametric equations
\begin{equation} \label{eq:trefoil}
\begin{split}
x &= r \frac{\sin(3 u)}{2 + \cos(v)},\\
y &= r \frac{\sin(u) + 2 \sin(2 u)}{2 + \cos\left(v + \frac{2\pi}{3}\right)},\\
z &= \frac{r}8 (\cos(u) - 2 \cos(2 u)) (2 + \cos(v)) \left(2 + \cos\left(v + \frac{2 \pi}{3}\right)\right),
\end{split}
\end{equation}
with $u,v\in[0,2\pi]$, $r>0$. We look at the two coordinate lines obtained with $v=0$ and $v = \pi$, $r=1$, and use the points given by $u=2\pi i/9$, $i\in\Z$, as initial data. We apply the scheme for 5 iterations with $\rho\in\{0,1,2\}$ for $v=0$ (top row) and $\rho\in\{0,2,6\}$ for $v=\pi$ (bottom row). We show in Figure \ref{fig:trefoil} the resulting 3D curves from front, lateral, and top views. We can see that the curves are smooth and the parameter $\rho$ governs the level of flexibility.

\section{Conclusions} \label{sec:conclusions}

In this study, we introduced a uniform non-linear non-stationary subdivision scheme for generating curves in $\R^n$, where $n\geq2$. It relies on annihilation operators to automatically estimate grid spacing and select an appropriate Lagrange-type scheme. As a result, a uniform scheme able to reproduce second-degree polynomials on a variety of non-uniform grids was obtained. Please refer to the Reproducibility Section for access to a MATLAB implementation.

The scheme is dependent on a positive parameter, which allows for control over both the level of flexibility and the variety of non-uniform grids on which second-degree polynomials are reproduced. We conducted numerical experiments to validate these findings, also revealing that the scheme is capable of generating visually appealing curves, that fits the data better than its analogous linear scheme, studied by \cite{DFH04}. Despite the subdivision scheme was defined from a very restricted property (applicable only to second-degree polynomial data), it produces visually appealing results even for general data.

The scheme has been defined in a non-stationary manner, which facilitated the proof of its $\cC^1$ convergence through similarities with a well-behaved subdivision scheme. We proposed a new class of subdivision schemes, to which our scheme belongs, and proved their convergence, and even $\cC^1$ continuity under a condition that is easily met. We established convergence in two distinct ways. Firstly, by leveraging the properties of quasilinear and asymptotically equivalent schemes. Secondly, by adapting a result from non-linear stationary subdivision theory to our non-stationary context.

For future research, we propose several avenues. Firstly, the current scheme exhibited stable behaviour during our numerical testing, meaning that slight modifications of the data induced small variations on the generated curve. However, it is worth noting that the piecewise definition of its rules is not continuous, and thus, stability is not guaranteed.
This scheme could be modified to ensure that its rules are Lipschitz functions, which could result in a stable scheme.
We believe this is achievable, although it might necessitate sacrificing certain properties, such as limiting its reproduction capability.

The simultaneous reproduction of parabolas on many non-uniform grid could be extended to higher-degree polynomials or to exponential polynomials. Another challenging task would be the extension of this work to the multivariate context.

The regularity was constrained by the regularity of the implicit mesh refinement, which was based on Chaikin's scheme. The adoption of a smoother and strictly monotonicity preserving scheme may be considered.

While the selection of the parameter $\rho$ is straightforward in practice to achieve good-looking results, finding its optimal value (in a certain sense) could be interesting. Rather than considering a constant parameter $\rho$, it could be defined as a function of the data, denoted as $\rho(f_{i-p},\ldots,f_{i+p+1})$. This would result in a parameter-less scheme that could potentially offer, for instance, improved adaptability or reduced distance between the generated curve and the data.

	In some CAD applications, the strict requirement of exact polynomial reproduction may be unnecessary, as the human eye often cannot perceive small approximation errors. This motivates the development of a simpler scheme that approximates the proposed scheme, potentially offering an efficient method for generating visually appealing curves.

	The proposed subdivision scheme is non-interpolatory, meaning that the resulting curve may not pass through the initial data points. Since this may be undesirable in some applications, a useful extension of this work would be the development of an interpolatory uniform scheme capable of reproducing non-uniform polynomial data.



Finally, we demonstrated $\cC^1$ regularity and also argued that it is not $\cC^2$. Nevertheless, the numerical experiments suggest that it could be curvature continuous. A comprehensive analysis of this topic could be conducted to reveal the conditions under which the generated curve is curvature continuous.

\section*{Acknowledgments}
This research has been supported by project CIAICO/2021/227 (funded by Conselleria de Innovación, Universidades, Ciencia y Sociedad digital, Generalitat Valenciana) and by grant PID2020-117211GB-I00 (funded by MCIN/AEI/10.13039/501100011033).

\section*{Reproducibility}

The MATLAB codes used for the numerical results in this paper, as well as a Wolfram Mathematica notebook file containing symbolic computations that complement the proofs presented in the manuscript, are available on Github: \url{https://github.com/serlou/uniform4nonuniform-subdivision}.

\section*{Declaration of Generative AI and AI-assisted technologies in the writing process}

During the preparation of this work the author used ChatGPT in order to improve readability and language. After using this service, the author reviewed and edited the content as needed and take full responsibility for the content of the publication.

\bibliography{biblio}

\end{document}